\documentclass[11pt]{amsart} \usepackage{latexsym, amssymb, stmaryrd}
\usepackage[T1]{fontenc}

\newtheorem{thm}{Theorem}[section]
\newtheorem{lemma}[thm]{Lemma}
\newtheorem{prop}[thm]{Proposition}
\newtheorem{cor}[thm]{Corollary}
\newtheorem{question}[thm]{Question}
\newtheorem{fact}[thm]{Fact}

\theoremstyle{definition}
\newtheorem{df}[thm]{Definition}

\theoremstyle{remark}
\newtheorem*{rmk}{Remark}
\newtheorem*{rmks}{Remarks}

\makeatletter

\def\dotminussym#1#2{%
  \setbox0=\hbox{$\m@th#1-$}%
  \kern.5\wd0%
  \hbox to 0pt{\hss\hbox{$\m@th#1-$}\hss}%
  \raise.6\ht0\hbox to 0pt{\hss$\m@th#1.$\hss}%
  \kern.5\wd0}
\newcommand{\dotminus}{\mathbin{\mathpalette\dotminussym{}}}
\renewcommand{\r}{\mathbb{R}}
\newcommand{\Z}{\mathbb{Z}}
\newcommand{\bC}{\mathbb{C}}

\newcommand{\curly}[1]{\mathcal{#1}}

\newcommand{\B}{\curly{B}}
\newcommand{\C}{\curly{C}}

\newcommand{\N}{\mathfrak{N}}
\newcommand{\n}{\mathbb{N}}
\newcommand{\cU}{\curly{U}}

\renewcommand{\S}{\curly{S}}\newcommand{\cR}{\mathcal{R}}
\newcommand{\cN}{\mathcal{N}}

\newcommand{\la}{\curly{L}}

\renewcommand{\to}{\rightarrow}

\newcommand{\e}{\epsilon}

\def \<{\langle}
\def \>{\rangle}
\def \z {{\mathbb Z}}
\def \*Z {{{^*}\Z}}
\def \((  {(\!(}
\def \)) {)\!)}

\def \m{\operatorname{M}}

\def \O{\operatorname{O}}

\def \tp{\operatorname{tp}}

\numberwithin{equation}{section}

\def \Th{\operatorname{Th}}
\def \R{\mathcal R}
\def \u{\mathcal U}

\def \c{\mathbb C}

\def \opsys{\operatorname{opsys}}
\def \ops{\operatorname{ops}}

\def \O{\mathcal O}
\def \nuc{\operatorname{nuc}}

\def \id{\operatorname{id}}

\def \cb{\operatorname{cb}}
\def \cS{\mathcal S}
\def \m{\mathcal M}
\def \N{\mathcal N}
\def \k{\mathcal K}

\allowdisplaybreaks[2]

\title[On Kirchberg's Embedding Problem]{On Kirchberg's Embedding Problem}
\author{Isaac Goldbring and Thomas Sinclair}
\thanks{Goldbring's work was partially supported by NSF grant DMS-1007144. Sinclair was supported by an NSF RTG Assistant Adjunct Professorship.}

\address {Department of Mathematics, Statistics, and Computer Science, University of Illinois at Chicago, Science and Engineering Offices M/C 249, 851 S. Morgan St., Chicago, IL, 60607-7045}
\email{isaac@math.uic.edu}
\urladdr{http://www.math.uic.edu/~isaac}

\address{Department of Mathematics, University of California Los Angeles, 520 Portola Plaza Box 951555, Los Angeles, CA 90095-1555}
\email{thomas.sinclair@math.ucla.edu}
\urladdr{http://www.math.ucla.edu/~thomas.sinclair}

\begin{document}

\begin{abstract}
Kirchberg's Embedding Problem (KEP) asks whether every separable C$^*$ algebra embeds into an ultrapower of the Cuntz algebra $\O_2$.  In this paper, we use model theory to show that this conjecture is equivalent to a local approximate nuclearity condition that we call \emph{the existence of good nuclear witnesses}.  In order to prove this result, we study general properties of existentially closed C$^*$ algebras.  Along the way, we establish a connection between existentially closed C$^*$ algebras, the weak expectation property of Lance, and the local lifting property of Kirchberg.  The paper concludes with a discussion of the model theory of $\O_2$.  Several results in this last section are proven using some technical results concerning tubular embeddings, a notion first introduced by Jung for studying embeddings of tracial von Neumann algebras into the ultrapower of the hyperfinite II$_1$ factor.
\end{abstract}
\maketitle

\tableofcontents

\newpage

\section{Introduction} The Cuntz algebra $\O_2$ is defined to be the universal C$^*$-algebra generated by two isometries $v_1, v_2$ satisfying the relation $v_1v_1^* + v_2v_2^* = 1$. It is a unital, separable, simple, nuclear, purely infinite C$^*$-algebra (we refer to such algebras in short as \emph{Kirchberg algebras}). In the 1990s Kirchberg obtained two remarkable theorems which established $\O_2$ as a central object of study in the K-theoretic classification of separable, nuclear C$^*$-algebras: $A\otimes \O_2 \cong \O_2$ for any Kirchberg algebra $A$; and (jointly with N. Phillips) that any separable, exact (e.g., nuclear) C$^*$-algebra is embeddable in $\O_2$ (this in fact characterizes exactness).

The techniques developed by Kirchberg and Kirchberg-Phillips to prove the $\O_2$ embedding theorem in fact suggest the possibility that \emph{every} separable C$^*$-algebra is embeddable into an ultrapower of $\O_2$, though there are many examples of non-exact, separable C$^*$-algebras also due to Kirchberg \cite{K}. This possibility was already noticed by Kirchberg and we refer to the question of whether every separable C$^*$-algebra is $\O_2^\omega$-embeddable as \emph{Kirchberg's Embedding Problem} (KEP). We remark that this problem should not be conflated with Kirchberg's conjecture that there is a unique C$^*$-norm on the algebraic tensor product C$^*(\mathbb F_\infty)\odot \text{C}^*(\mathbb F_\infty)$, which is known to be equivalent to Connes' Embedding Problem \cite{K}.

The purpose of this paper is to investigate KEP from the perspective of the continuous model theory of C$^*$-algebras. The development of the model theory of C$^*$-algebras in general, and nuclear (or exact) C$^*$-algebras specifically, is in its very early stages, so many results obtained in this work are aimed at developing the foundational aspects of this theory. The results and ideas detailed below are complementary to (and borrow from) a larger systematic treatment of the model theory of nuclear C$^*$-algebras by Farah, et al.\ \cite{FHRTW} which is forthcoming. One important feature of our approach is that we emphasize the importance of the model theory of C$^*$-algebras as operator spaces, though this is also implicit in \cite{FHRTW}. In appendix B we have included a treatment of operator systems and operator spaces in the context of continuous logic as an exposition of this material has not appeared so far in the literature.

We now describe the structure of the paper, pointing out the main results in the order they appear. Aside from the introduction (first section), the paper is divided into four main sections and three appendices. As a means for preparing the groundwork for discussing KEP, the second section of the paper deals with general properties of existentially closed (e.c.) unital C$^*$-algebras. The first main result is that an exact, unital, e.c.\   C$^*$-algebra is nuclear (Theorem \ref{exactec}). Using results from \cite{FHRTW} and Kirchberg's ``$A\otimes \O_2$'' theorem, it follows that $\O_2$ is the only possible separable, exact, unital, e.c.\ C$^*$-algebra (Proposition \ref{onlypossibility}). As will be demonstrated in section 3, KEP is equivalent to whether $\O_2$ is actually existentially closed. We further show that all e.c.\ C$^*$-algebras have trivial K-theory. The second section ends with a discussion of the relationship between existential closedness and Lance's weak expectation property (WEP).

The main results on KEP are established in the third section. The first main result is that KEP has a positive solution if and only if $\O_2$ is existentially closed which, by the results described in section 2, is true if and only if there is a unital, exact, e.c.\ C$^*$-algebra (Theorem \ref{KEPequiv}). This is the first key step towards the main result of the paper: KEP has a positive solution if and only if every satisfiable condition admits \emph{good nuclear witnesses} (Theorem \ref{witnesses}). In operator algebraic parlance, the latter condition roughly means that any C$^*$-algebra can be locally modelled by operators in $\mathcal B(H)$ which are approximately nuclear in $\mathcal B(H)$. The hard direction of the proof of this result rests on a finer analysis of an ``omitting types'' characterization of nuclearity as developed in \cite{FHRTW} combined with a general Omitting Types Theorem from the theory of continuous model theoretic forcing to produce an existentially closed C$^*$-algebra from the collection of good nuclear witnesses, whence KEP follows by Theorem \ref{KEPequiv}.

The fourth and fifth sections of the paper are concerned with results on the general model theory of $\O_2$. The fourth section establishes some technical result on tubularity of embeddings of C$^*$-algebras, a concept first introduced in the context of von Neumann algebras by Jung \cite{Jung}. In the fifth and final section, it is established, among other results, that $\O_2$ is the prime model of its theory, that KEP implies that $\Th(\O_2)$ is not model complete nor does the theory of C$^*$ algebras admit a model companion, and that $\O_2$ is the only so-called \emph{stably presented} model of its theory. We leave it as open questions whether $\Th(\O_2)$ is model complete or admits an exact, non-nuclear model ($\O_2$ is the only nuclear model of $\Th(\O_2)$). It is observed though that at least one of the two questions must have a negative answer (Proposition \ref{goodnews}).

For the sake of completeness, three appendices are included. The first discusses the general theory of continuous model theoretic forcing towards establishing a technical result needed in the proof of Theorem \ref{witnesses}: that generic models of $\forall\exists$-axiomatizable theories are e.c.\ models of the theory. The second appendix gives an axiomatization of operator spaces and operator systems in continuous logic. The last appendix contains the proof of a result, due to Martino Lupini, on the definability of the operator norm on the matrix algebras $M_n(A)$ over a C$^*$-algebra $A$.

We assume that the reader is familiar with the model-theoretic treatment of operator algebras as presented in \cite{MTOA2}.  A reader looking for a quicker introduction to the model theory of operator algebras can consult the introduction of \cite{ecfactor} (which only treats von Neumann algebras).

In this paper, $\omega$ always denotes a nonprincipal ultrafilter on the natural numbers.  In the presence of a nonprincipal ultrafilter $\omega$, the phrase ``for almost all $i$, $P$(i) holds'' will mean that the set of $i$ for which $P(i)$ holds belongs to the ultrafilter $\omega$.

\subsection{Acknowledgements} We would like to thank Ilijas Farah, Bradd Hart, and David Sherman for useful discussions on the subject of this paper.  We would also like to thank Gilles Pisier for pointing us to the paper \cite{jungepisier}, allowing us to improve on an earlier draft of the paper.

\section{Existentially closed C$^*$ algebras}

\subsection{General properties}

Suppose that $A$ and $B$ are C$^*$ algebras and $i:A\to B$ is an embedding of C$^*$ algebras.  We say that $i$ is an \emph{existential embedding} if, for any quantifier-free formula $\varphi(v,w)$, where $v$ and $w$ are tuples of variables, any $n\geq 1$, and any tuple $a$ from $A$, we have
$$\inf\{\varphi(a,b) \ : \ b\in A_{\leq n}\}=\inf\{\varphi(i(a),c) \ : \ c\in B_{\leq n}\}.$$
If $A\subseteq B$ and the inclusion is an existential embedding, we say that $A$ is \emph{existentially closed in $B$}.  We say that $A$ is \emph{existentially closed} (e.c.) if $A$ is existentially closed in $B$ for any $B$ containing $A$.

In \cite{MTOA2}, the authors present the language for arbitrary (i.e., not necessarily unital) C$^*$ algebras.  For our purposes, this is a bit troublesome:

\begin{prop}
Suppose that $A$ is an e.c.\   $C^*$ algebra in the language for $C^*$ algebras that does not require a unit.  Then $A$ is not unital.
\end{prop}

\begin{proof}
Suppose that $A$ is unital.  Then for any $B\supseteq A$ we have that $B$ is unital and $1_A=1_B$.  Indeed, if $x\in B$ and $r:=d(1_A\cdot x,x)>0$, then $(\inf_x(r\dotminus d(1_A\cdot x,x)))^B=0$; since $A$ is e.c., there is $x\in A$ such that $d(1_A\cdot x,x)>\frac{r}{2}$, a contradiction.

Now such a phenomenon cannot happen: just take a nonunital $C^*$ algebra $M$ and let $B:=A\oplus M$.
\end{proof}

\emph{Thus, in the rest of this paper, unless otherwise stated, we work in the language of \emph{unital} $C^*$ algebras and all C$^*$ algebras are required to be unital}.

In the rest of this subsection, we outline some of the basic C$^*$ algebraic properties that any e.c.\  C$^*$ algebra must have. As the attentive reader may begin to suspect, existential closure imposes a myriad of strong structural properties on a C$^*$-algebra. Therefore, before proceding to establish some of these properties, we point out that existential closure is well worth studying as separable e.c.\ C$^*$-algebras are in some sense generic among all separable C$^*$-algebras. To be more precise, the set of separable e.c.\ C$^*$-algebras in dense in the the set of all separable C$^*$-algebras in the ``logic topology'' described in section 2.2 of \cite{usvyatsov}: see Fact 4.2 therein. In fact a random C$^*$-algebra is e.c.\ for any reasonable choice of probability measure on the space of separable C$^*$-algebras equipped with the logic topology by Lemma 4.7 in \cite{usvyatsov}. Moreover the separable e.c.\ C$^*$-algebras form a universal class:

\begin{prop} Every separable C$^*$-algebra is a (unital) subalgebra of an existentially closed C$^*$-algebra.
\end{prop}

\begin{proof} For a proof of this result in the context of classical logic see Lemma 3.5.7 in \cite{chang-keisler}: the proof for continuous logic is similar. See the proof of Fact 2.8 in \cite{usvyatsov} for more details.
\end{proof}

Occasionally, we may choose an embedding of a given C$^*$-algebra A into some e.c.\ C$^*$-algebra $B$ which we refer to as an ``existential closure'' of $A$ (there is no apparent reason why there is a unique or even minimal such one).

\begin{prop}\label{uncountable} There are uncountably many pairwise non-isomorphic separable e.c.\ C$^*$-algebras.
\end{prop}

\begin{proof} It follows from \cite[Proposition 2.6]{jungepisier} that there is no separable C$^*$-algebra into which all separable C$^*$-algebras embed. Thus there cannot be countably many isomorphism classes of separable e.c.\ C$^*$-algebras, lest we could construct such a universal C$^*$-algebra as a tensor product of representatives from each class.
\end{proof}

Let $A$ be a (unital) C$^*$ algebra.  We say that $A$ has the \emph{strong Dixmier property} (SDP) if, for every $a_1,\ldots,a_n\in A_+$ and every $\epsilon>0$, there are unitaries $u_1,\ldots,u_k\in A$ and real numbers $r_1,\ldots,r_n$ with $\|a_j\|\leq 4|r_j|$ and $$\|\frac{1}{k}\sum_{i=1}^k u_ia_ju_i^*-r_j\cdot 1\|<\epsilon$$ for each $j=1,\ldots,n$.  (While there is a general consensus on what the Dixmier property means, there does not seem to be a consensus on what the strong Dixmier property means.  In particular, this is not what Blackadar \cite[III.2.5.16]{Black} calls the strong Dixmier property.  Instead, we follow Kirchberg's terminology from \cite{K}.)

\begin{fact}
If $A$ has the SDP, then $A$ is simple.
\end{fact}

\begin{proof}
Suppose that $I$ is a nonzero (closed) ideal in $A$.  Fix $a\in I_+$ with $\|a\|=4$.  Take unitaries $u_1,\ldots,u_k\in A$ and a real number $r$ such that, setting $Y:=\frac{1}{k}\sum_{i=1}^k u_iau_i^*$, we have $\|Y-r\cdot 1\|<\frac{1}{2}$ and $|r|\geq 1$.  Then $\|\frac{1}{r}Y-1\|<\frac{1}{2}$, whence $\frac{1}{r}Y$ is invertible.  But $\frac{1}{r}Y\in I$, whence $I=A$.
\end{proof}

\begin{prop}
An e.c.\   (unital) C$^*$ algebra has SDP.  In particular, an e.c.\ unital C$^*$ algebra is simple.
\end{prop}

\begin{proof}
First, given any finite tuple $a$ from $A$, there is a separable elementary substructure $A_0$ of $A$ containing $a$; if $A$ is e.c., then so is $A_0$.  It thus suffices to consider the case that $A$ is separable.  The proof of \cite[Corollary 3.5]{K} shows that every unital separable C$^*$ algebra embeds into a separable (unital) C$^*$ algebra with SDP.   Thus, we can embed $A$ into $B$ separable with SDP.  Given any $a_1,\ldots,a_n\in A_+$ and $\epsilon>0$ there are real numbers $r_1,\ldots,r_n$ such that $\inf\varphi(\bar a,\bar y)^B=0$, where $\varphi(\bar a,\bar y)$ is the formula
$$\max\left(\max_{1\leq j\leq n}\|\sum_{i=1}^k y_ia_jy_i^*-r_j\cdot 1\|,\sum_{i=1}^k \|y_i^*y_i-1\|\right).$$
By existential closedness of $A$ and functional calculus, we may find unitaries in $A$ witnessing that $A$ has the SDP.
\end{proof}

\begin{cor}
C$^*(\mathbb F_2)$ and C$^*(\mathbb F_\infty)$ are not existentially closed.
\end{cor}

\begin{prop}
If $A$ is an e.c.\   C$^*$ algebra, then every automorphism of $A$ is approximately inner.
\end{prop}

\begin{proof}
This was proven in \cite{ecfactor} for tracial von Neumann algebras so we only sketch the proof.  If $\alpha$ is an automorphism of $A$, then setting $B:=A\rtimes_\alpha \mathbb{Z}$, we see that $\alpha$ is implemented by a unitary $u$ in $B$.  Thus, for any $x_1,\ldots,x_n\in A$, there is a unitary $u\in B$ that conjugates $x_i$ to $\alpha(x_i)$ for each $i=1,\ldots,n$.  Since $A$ is e.c., one concludes that there is an ``almost'' unitary in $A$ that conjugates each $x_i$ close to $\alpha(x_i)$; by functional calculus, this almost unitary can be correct to an actual unitary while maintaining that the conjugate of $x_i$ and $\alpha(x_i)$ remain close.
\end{proof}

\begin{fact}[\cite{FHRTW}]\label{O2stableaxiom}
The following properties are $\forall\exists$-axiomatizable properties of separable $C^*$ algebras:
\begin{enumerate}
\item $\O_2$-stable;
\item simple and purely infinite.
\end{enumerate}
\end{fact}

Although the above result is proven in \cite{FHRTW}, we repeat their proof of $\O_2$-stability here for the current reader's convenience.  For $A$ separable, $A$ is $\mathcal O_2$-stable if and only if $\mathcal O_2$ embeds into $A'\cap A^\omega$; see, for example, \cite[Theorem 7.2.2]{Ror}.  Now notice that $\mathcal O_2$ embeds into $A'\cap A^\omega$ if and only if $A$ satisfies the axioms $$\sup_{\vec x}\inf_{y_1,y_2}(\Sigma \|[x_i,y_j]\|+d(y_1^*y_1,1)+d(y_2^*y_2,1)+d(y_1y_1^*+y_2y_2^*,1))=0,$$ where we have one such axiom for all possible lengths of the finite tuples $\vec x$ appearing in the displayed formula.

\begin{cor}
A separable e.c.\   $C^*$ algebra is $\mathcal O_2$-stable.
\end{cor}

\begin{proof}
By general model theory, if $P$ is an $\forall\exists$-axiomatizable property of separable C$^*$ algebras and every separable C$^*$ algebra embeds into a separable C$^*$ algebra with property P, then a separable e.c.\   C$^*$ algebra has property P.  Clearly any separable C$^*$ algebra embeds into a separable $\O_2$-stable C$^*$ algebra:  simply tensor with $\O_2$.
\end{proof}

From the above corollary, we also see that an e.c.\   C$^*$ algebra does not admit a quasitrace.  Indeed, if an e.c.\   C$^*$ algebra admitted a quasitrace, then since it is also simple, it would be stably finite (see, for example, \cite[Theorem 2.2]{RorICM}), contradicting the previous corollary.

\



Recall that if $A$ is a C$^*$ algebra, then the \emph{theory of $A$} is the function $\Th(A)$ which maps a sentence $\sigma$ to its truth value $\sigma^A$.  We say that $A$ and $B$ are \emph{elementarily equivalent}, denoted $A\equiv B$, if $\Th(A)=\Th(B)$; in this case, we also say that $A$ is a model of $\Th(B)$ (and, of course, that $B$ is a model of $\Th(A)$).

We say that a C$^*$ algebra $A$ is an existentially closed model \emph{of its theory} if it satisfies the above definition of being existentially closed but only for extensions $B$ that are elementarily equivalent to $A$.  We define the \emph{universal theory of $A$} to be the function $\Th_\forall(A)$ which is simply the restriction of $\Th(A)$ to the set of universal sentences.  (One can, of course define $\Th_\exists(A)$ and $\Th_{\forall\exists}(A)$ similarly.)  It is a basic fact of model theory that if $A$ and $B$ are separable C$^*$ algebras, then $\Th_\forall(A)\leq \Th_\forall(B)$ (as functions) if and only if $A$ is embeddable into an ultrapower of $B$.  Thus, $A$ is an e.c.\   model of its theory if and only if it satisfies the above definition of being existentially closed but only for extensions $B$ that are embeddable into an ultrapower of $A$.  Finally, recall that an embedding between C$^*$ algebras $f:A\to B$ is said to be \emph{elementary} if, for any formula $\varphi(x)$, where $x$ is an $m$-tuple of variables, and any $m$-tuple $a$ from $A$, we have $\varphi(a)^A=\varphi(f(a))^B$.  If $A\subseteq B$ is such that the inclusion map is elementary, we say that $A$ is an \emph{elementary substructure of $B$}, denoted $A\preceq B$.

\begin{lemma}\label{ssa}
Suppose that $A$ is a separable C$^*$ algebra such that every embedding of $A$ into $A^\omega$ is elementary.  Then $A$ is an existentially closed model of its theory.  In particular, $\O_2$ is an existentially closed model of its theory.
\end{lemma}

\begin{proof}
Suppose that $B$ is embeddable in $A^\omega$, say $j:B\to A^\omega$, $\varphi(v,w)$ is a formula, and $a\in A$, we have
$$(\inf_w \varphi(a,w))^A\geq (\inf_w \varphi(a,w))^B\geq (\inf_w \varphi(j(a),w))^{A^\omega}=(\inf_w \varphi(a,w))^A.$$

It is well known that every embedding of $\O_2$ into its ultrapower is unitarily conjugate to the diagonal embedding, whence elementary.  (The main ideas behind the proof can be found at the end of Section 5.)
\end{proof}

In \cite{FHRTW}, it is shown that any strongly self absorbing C$^*$ algebra (see \cite{TW}) satisfies the assumption of the previous lemma, whence any strongly self absorbing C$^*$ algebra is an e.c.\   model of its theory.

\subsection{Connection with nuclearity and exactness}

In this subsection, we investigate whether or not there is a connection between existential closedness on the one hand and exactness and/or nuclearity on the other hand.  A first observation:

\begin{fact}
Not every e.c.\   $C^*$ algebra is exact.
\end{fact}

\begin{proof}
Since being exact is preserved under substructure and every C$^*$ algebra embeds into an e.c.\ one, we would then have that every C$^*$ algebra is exact, a contradiction.
\end{proof}

Let us be less ambitious:

\begin{question}
Is there an exact e.c.\ $C^*$ algebra?  Is there a nuclear e.c.\ C$^*$ algebra?
\end{question}

It turns out these are the same question:

\begin{thm}\label{exactec}
If $A$ is an exact e.c.\ $C^*$ algebra, then $A$ is nuclear.
\end{thm}

Theorem \ref{exactec} follows from the following more general statement:

\begin{thm}\label{exactec2}
Suppose that $f:A\to B$ is an existential nuclear embedding.  Then $A$ is nuclear.
\end{thm}

\begin{proof}
For the sake of simplicity, let us suppose that $f$ is an inclusion $A\hookrightarrow B$ and that $A$ is e.c.\   in $B$.  Once again, for simplicity assume that $A$ is separable (although the general case is no more difficult.)  Let $(a_n)$ be an enumeration of a dense subset of $A$ and let $\phi_n:A\to M_{k_n}$ and $\tau_n:M_{k_n}\to B$ be completely positive contractions such that $$\max_{1\leq j\leq n}\|(\tau_n\circ \phi_n)(a_j)-a_j\|\leq \frac{1}{n}.$$  For each $j$, let $\lambda_{k,l}^{j,n}$ denote the coordinates of $\phi_n(a_j)$ with respect to the standard basis of $M_{k_n}$.  Fix a finite $\delta$-net $F\subseteq (M_{k_n})_{1}$, for $\delta$ sufficiently small.  Consider the formulae $\sigma_1:=\|w^*w-u\|,$ $\sigma_2:=\max_{1\leq j\leq n}\|\sum \lambda_{k,l}^{j,n}u_{kl}-a_j\|\dotminus \frac{1}{n},$ and $\sigma_3:=\max_{\eta_{k,l} \in F}(\|\sum \eta_{k,l}u_{kl}\|\dotminus \|\sum \eta_{k,l}e_{k,l}\|).$  We should mention that in $\sigma_1$, we are using the max norm on $M_{k_n}(B)$ and not the operator norm.  Then the formula $\inf_{u,w\in M_{k_n}(B)}  \max(\sigma_1,\sigma_2,\sigma_3)$ evaluates to $0$ in $B$.


Since the above formula only mentions parameters from $A$, it must also evaluate to $0$ in $A$.  By functional calculus, this allows us to choose $u\in M_n(B)_+$ such that, defining $\psi_n:M_{k_n}\to A$ by $\psi_n(e_{k,l}):=u_{k,l}$, we have that $$\|(\psi_n\circ \phi_n)(a_j)-a_j\|\leq \frac{2}{n}$$ for $j=1,\ldots,n$.  Moreover, if $\delta$ was chosen sufficiently small (and if the $u$ was chosen to make the $\inf$ sufficiently small in $A$), we get that $\psi_n$ is $\rho_n$-Lipshitz, where $\lim_{n\to \infty}\frac{1}{\rho_n}=1$.  By rescaling $u$ by $\frac{1}{\rho_n}$ and choosing $\rho_n$ sufficiently small, we get a positive $u$ that defines $\psi_n$ that is contractive and for which the error $\|(\psi_n\circ \phi_n)(a_j)-a_j\|\leq \frac{3}{n}$ for $j=1,\ldots,n$.  By Choi's Theorem (see \cite[Theorem 3.14]{Paulsen}), $\psi_n$ is completely positive.  It follows that the identity map on $A$ is nuclear, whence $A$ is nuclear.
\end{proof}

In the above proof, observe that the only use of the algebra structure was to express positivity.  Thus, we really have proven:

\begin{thm}\label{exactecopsys}
Suppose that $f:A\to B$ is a nuclear embedding that is existential in the language of operator systems.  Then $A$ is nuclear.
\end{thm}

Since there has yet to be any exposition of operator spaces and operator systems in continuous logic appearing in the literature, we take the opportunity to present this in Appendix B.

We would like to offer an alternative proof of the previous theorem.  Without loss of generality, assume that $A$ is a subalgebra of $B$ such that the inclusion is nuclear and such that $A$ is e.c.\   in $B$ in the language of operator systems.  Again, for simplicity, we assume that $A$ is separable.  Let $(a_i)$ enumerate a countable dense subset of $(A^+)_1$. Fix $j,\delta$ and choose $n$ and u.c.p.\ maps $\phi: A\to M_n(\bC)$ and $\psi: M_n(\bC)\to B$ so that the following condition $(\ast)$ holds: $\max_{i\leq j} \|\psi\circ\phi(a_i) - a_i\|<\delta$.

Recall that any positive matrix $x\in M_n(\bC)$ can be written as a sum of $n$ matrices of the form $v^*v$ where $v = (v_1,\dotsc,v_n)$. (Use the diagonalization.) Write $\phi(a_i)$ as $\sum_{k=1}^n (v_k^i)^*v_k^i$. Examining the operator space duality ${\rm CP}(M_n(\bC), B) \leftrightarrow M_n(B)^+$ we see that condition $(\ast)$ is the same as \[\inf_{\Psi\in M_n(B)^+}\max_{i\leq j}\|\sum_{k=1}^n v_k^i \Psi (v_k^i)^* - a_i\|<\delta.\] Since $A$ is e.c.\   in $B$ in the language of operator systems, we can find such $\Psi_n \in M_n(A)^+$ instead. This concludes the proof.



The hypothesis in Theorem \ref{exactecopsys} really is a weakening of the hypothesis in Theorem \ref{exactec2}:

\begin{prop}\label{ecchangecategories}
Suppose that $A$ and $B$ are unital C$^*$ algebras such that $A$ is a unital subalgebra of $B$.  If $A$ is e.c.\   in $B$ in the language of C$^*$ algebras then $A$ is e.c.\   in $B$ in the language of operator systems.
\end{prop}

\begin{proof}
It was already proven in \cite{FHRTW} that if $A$ is elementary in $B$ then $M_n(A)$ is elementary in $M_n(B)$, the crucial point being that the operator norm on $M_n(A)$ is a definable predicate in $A$, a fact due to Martino Lupini and proven here in Appendix C.  A more careful analysis of this fact yields the proof of the current proposition.  In Appendix C, we prove the following:  let $A$ be a C$^*$ algebra and let
$$X_{n,A}:=\{(\vec x,\vec y)\in A_1^{2n} \ : \ \max\left(\left\|\sum x_i^*x_i\right\|,\left\|\sum y_i^*y_i\right\|\right)\leq 1\}.$$   Then for $(a_{ij})\in M_n(A)$ we have $\|(a_{ij})\|=\sup_{(\vec x,\vec y)\in X_{n,A}}\|\sum x_i^*a_{ij}y_j\|$.  Moreover, the relation defining $X_{n,A}$ is stable, whence $X_{n,A}$ is uniformly definable in all C$^*$ algebras in the sense that there is a formula $\psi(\vec x,\vec y)$ such that $\psi(\vec x,\vec y)^A$ is the distance between $(\vec x,\vec y)$ and $X_{n,A}$ in all C$^*$ algebras $A$.  (See \cite[Lemma 2.1]{AF} for a proof of the preceding fact.)  By the proof of \cite[Theorem 9.17]{BBHU}, it follows that there is a formula (as opposed to an arbitrary definable predicate) $\Phi_n((z_{ij}))$ such that, for all C$^*$ algebras $A$ and all $(a_{ij})\in M_n(A)$ we have $\Phi_n((a_{ij}))^A=\|(a_{ij})\|$.  Finally, it remains to notice that $\Phi_n$ is approximable by \emph{quantifier-free} formulae, uniformly in all C$^*$ algebras, whence any existential statement about $M_n(A)$ in the language of operator systems can be approximated by existential statements about $M_n(A)$ in the language of C$^*$ algebras, finishing the proof of the proposition.  Indeed, by a standard quantifier-elimination test (see \cite[Proposition 13.2]{BBHU}), it suffices to show that whenever $A$ is a C$^*$ algebra that is a common subalgebra of the C$^*$ algebras $B$ and $C$, then for any $(a_{ij})\in M_n(A)$, we have $\Phi_n((a_{ij}))^B=\Phi_n((a_{ij}))^C$, which clearly holds as the operator norm on $M_n(A)$ is the restriction of both the operator norm on $M_n(B)$ and the operator norm on $M_n(C)$.
\end{proof}

There is something a bit strong about asking that a C$^*$ algebra be existentially closed as an operator system for it would require, by definition, that $A$ be existentially closed in all operator systems in which $A$ admits a complete order embedding.  In particular, it would imply that any embedding of $A$ into a C$^*$ algebra that is merely a complete order embedding is an existential embedding.  This requirement is so strong that the first author and Martino Lupini were able to prove in \cite{goldlupini} that if $A$ is an existentially closed C$^*$ algebra, then $A$ is \emph{not} existentially closed as an operator system.  Thus, it seems appropriate to allow for an intermediate notion.

\begin{df}\label{semi}
We say that a C$^*$ algebra $A$ is \emph{semi-existentially closed as an operator system} (resp. \emph{semi-existentially closed as an operator space}) if, for any unital inclusion $A\subseteq B$ of C$^*$ algebras, any existential formula $\varphi(x)$ in the language of operator systems (resp. in the language of operator spaces), and any tuple $a$ from $A$, we have $\varphi(a)^A=\varphi(a)^B$.
\end{df}

We can thus restate Theorem \ref{ecchangecategories} as follows:  if a C$^*$ algebra is existentially closed (as a C$^*$ algebra), than it is semi-existentially closed as an operator system.

Theorem \ref{exactec} tells us that we should focus our attention on the search for an e.c.\   nuclear C$^*$ algebra.  If we search for a separable nuclear C$^*$ algebra, we know where to search:

\begin{prop}\label{onlypossibility}
If $A$ is a separable unital e.c.\   nuclear $C^*$ algebra, then $A\cong \mathcal O_2$.
\end{prop}

\begin{proof}
Since $A$ is e.c., we have $A\otimes \mathcal O_2\cong A$.  However, because $A$ is simple and nuclear, by Kirchberg's ``$A\otimes \O_2$ Theorem'' (see \cite[Theorem 7.1.2]{Ror}), we have that $A\otimes \mathcal O_2\cong \mathcal O_2$, whence we have $A\cong \mathcal O_2$.
\end{proof}

Thus even though there are uncountably many separable e.c.\ C$^*$-algebras, at most one of them, $\O_2$, can be exact. The obvious question is:  is $\O_2$ e.c.?  This turns out to be equivalent to KEP; see Theorem \ref{KEPequiv}.

We just saw that $\O_2$ is the only possible nuclear (even exact) separable e.c.\   nuclear C$^*$ algebra.  The next result says that $\O_2$ is the only possible separable e.c.\   C$^*$ algebra satisfying some other properties.

\begin{prop}\label{flip} Let A be unital separable e.c.\   C$^*$ algebra.  Then the following are equivalent:

\begin{enumerate}
\item $A \cong \O_2$;
\item $A \cong A\otimes_{\operatorname{min}} A$;
\item The ``flip'' $(a\otimes b)\to (b\otimes a)$ is an approximately inner automorphism of $A\otimes_{\operatorname{min}} A$;
\item The embeddings $A\otimes 1$ and $1\otimes A$ into $A\otimes_{\operatorname{min}} A$ are approximately unitarly equivalent.
\end{enumerate}
\end{prop}

\begin{proof}
(1)$\Rightarrow$(2) is trivial.  (2)$\Rightarrow$(3) holds since $A\otimes_{\operatorname{min}} A$ would be e.c., hence any automorphism would be approximately inner.  (3)$\Rightarrow$(4) is clear.  (4)$\Rightarrow A$ is nuclear (\cite[Lemma 3.10]{KP}), which implies (1) by the previous proposition.
\end{proof}

In fact, the above argument shows that the only possible e.c.\   ``tensor square'' is $\O_2$, that is, the only possible separable e.c.\   $B$ such that $B\cong A\otimes A$ for some $A$ is $\O_2$.

\subsection{K-theory}\footnote{The remaining two subsections are not needed in connection with the material on Kirchberg's Embedding Problem appearing in the next section.}

In this subsection, we show that e.c.\   C$^*$ algebras have trivial $K$-theory.  Since e.c.\ C$^*$ algebras are $\O_2$-stable, the results of this subsection follow from the fact (pointed out to us by Chris Phillips) that $A\otimes \O_2$ has trivial K-theory for any C$^*$ algebra $A$.  Since this aforementioned fact uses the K\"unneth formula for C$^*$ algebras (see \cite{schochet}) together with the fact that $\O_2$ belongs to the so-called ``bootstrap class'' (see \cite[Section 4.2]{Ror}), we prefer to give completely elementary arguments, which further highlight the role of existential closedness in connection with familiar C$^*$-algebraic concepts.

\begin{prop}
If $A$ is an e.c.\   $C^*$ algebra, then $K_0(A)=\{0\}$.
\end{prop}

\begin{proof}
Suppose that $A$ is an e.c.\   C$^*$ algebra and $p\in P(M_m(A))$ and $q\in P(M_n(A))$ with $m\leq n$.  We claim that $p$ and $q$ are stably equivalent, so that $K_0(A)^+=\{0\}$.  Suppose that $A$ is concretely represented as a C$^*$ subalgebra of $\B(H)$ where $H$ is an infinite-dimensional Hilbert space.  Since $K_0(\B(H))^+=\{0\}$, we know that there is $r\in P(M_k(\B(H)))$ such that $p\oplus r\oplus 0_{n-m}\sim q\oplus r$, that is, there is $x\in M_{n+k}(A)$ such that
$$p\oplus r\oplus 0_{n-m}=x^*x, \ \ q\oplus r=xx^*.$$  Thus, in $\B(H)$, the following sentence evaluates to $0$:
$$\inf_{r\in M_k}\inf_{x\in M_{n+k}}\max(d(r,r^*),d(r,r^2),d(p\oplus r\oplus 0_{n-m},x^*x),d(q\oplus r,xx^*)).$$

Once again, the $d$ symbols refer to the max norm on the appropriate matrix algebras and not the operator norm.  In connection with the previous sentence, it is useful to note that, for $M_k(A)$, the operator norm is bounded by $\sqrt{k}$ times the max norm.

Fix $\epsilon_0$ such that, for all $C^*$ algebras $B$ and all projections $p,q\in B$, we have that $p$ and $q$ are Murray-von Neumann equivalent if and only if there is $b\in B$ such that $\max(\|p-b^*b\|,\|q-bb^*\|)<\epsilon_0$.  Since $A$ is e.c., using functional calculus, we may find $r\in P(M_k(A))$ and $x\in M_{n+k}(A)$ such that $d(p\oplus r\oplus 0_{n-m},x^*x),d(q\oplus r, xx^*)<\frac{\epsilon_0}{\sqrt{n+k}}$, whence they are within $\epsilon_0$ in operator norm.  It follows that $p$ and $q$ are stably equivalent in $A$.
\end{proof}

\begin{prop}
If $A$ is an e.c.\  C$^*$ algebra, then $K_1(A)=\{0\}$.
\end{prop}

\begin{proof}
For $m\geq 1$, let $U_m(A)$ denote the unitary group of $M_m(A)$.  Fix $u\in U_m(A)$ and $v\in U_n(A)$.  We want to show that $u\sim_1 v$.  Since $\B(H)$ has trivial $K_1$, there is $k\geq \max(m,n)$ such that $u\oplus 1_{k-m}\sim_h v\oplus 1_{k-n}$, say witnessed by the path $\alpha:[0,1]\to U(H^k)$.  By uniform continuity of the path, there is $j\geq 1$ such that \[\inf_{y_1,\ldots,y_j\in M_k(B(H))}\varphi(u,v,\vec y)\] evaluates to $0$ in $B(H)$, where, setting \[\varphi_1(\vec y):=\max_{1\leq i\leq j}\max(d(y_i^*y_i,1),d(y_iy_i^*,1))\] and \[\varphi_2(\vec y):=\max_{1\leq i<j}(d(y_i,y_{i+1})\dotminus \frac{1}{2}),\] $\varphi(u,v,\vec y)$ is the formula
$$\max(\varphi_1(\vec y),\varphi_2(\vec y),d(u\oplus 1_{k-m},y_1)\dotminus \frac{1}{2},d(y_j,v\oplus 1_{k-n})\dotminus \frac{1}{2}).$$  Once again, the metric is the max metric.  Since $A$ is e.c., the value of the sentence is $0$ in $A$ as well.  Since the operator norm on $M_k(A)$ is bounded by $\sqrt{k}$ times the max norm and since the unitary group of a C$^*$ algebra is definable, there are actual unitaries making this $\inf$ as small as we want.  Since elements of the unitary group that are sufficiently close are automatically homotopy equivalent, we get that $u\sim_1 v$ in $A$.  It follows that $K_1(A)=\{0\}$.
\end{proof}

One can play a similar game to show that an e.c.\ C$^*$ algebra has trivial Cuntz semigroup.


\subsection{Connection with the Weak Expectation Property}

We recall the following

\begin{df} Let $A\subset B$ be a unital inclusion of C$^*$-algebras. We say that $B$ is \emph{weakly injective relative to A} if there exists a contraction $\phi: B\to A^{**}$ such that $\phi |_A = \id_A$ (the map $\phi$ is known as a weak conditional expectation). We say that $A$ has the \emph{weak expectation property} (WEP) of Lance if every unital extension $A\subset B$ of $A$ is weakly injective relative to $A$.
\end{df}

We remark that the canonical extension $\phi^{**}: B^{**}\to A^{**}$ can be shown to be a norm one projection from $B^{**}$ onto $A^{**}$ (see \cite{K}, p. 459). By Tomiyama's theorem, $\phi^{**}$ is u.c.p.\ whence so is $\phi$.

An operator system $X\subset \B(H)$ is \emph{weakly injective} (or \emph{almost injective} in the terminology of \cite{CE}) if there is a u.c.p.\ extension $\phi: \B(H)\to \overline{X}$ of the identity $\id: X\to X$, where $\overline{X}$ is the weak closure of $X$. The following proposition is heavily inspired by an argument of Choi and Effros (see \cite[Theorem 3.4]{CE}).

\begin{prop}\label{ecinj} Let $X\subset \B(H)$ be an operator system. If $X$ is e.c.\ as an operator system in $\B(H)$, then $X$ is weakly injective.
\end{prop}

\begin{proof} Let $\C_n$ be the positive cone in $M_n(\bullet)$. Consider the existential formulae \[\phi_{n,k,\sigma}(a,b) = \inf_{x\in \C_n}\|[a_{ij} + \sigma_{ij}^1b_1 + \dotsb + \sigma_{ij}^k b_k] - x\|\] where $[a_{ij}]\in M_n(\bullet)$ is a self-adjoint matrix and $\sigma =(\sigma^1,\dotsc,\sigma^k)\in M_n(\bC)^k$ is self-adjoint. For $b = (b_1,\dotsc,b_k)\in \B(H)^k$ self-adjoint, operator systems $X\subset Y\subset \B(H)$, and $b'\in Y^k$, note that the linear map $$\eta: X+ \bC b_1 + \dotsb + \bC b_k \to Y, \quad \eta(x + \sum_l \lambda_l b_l) := x + \sum_l \lambda_l b_l'$$ is u.c.p.\ if and only if, for every self-adjoint $a\in M_n(X)$ and every self-adjoint $\sigma\in M_n(\mathbb C)^k$, we have $\phi_{n,k,\sigma}(a,b)^{\B(H)}=0$ implies $\phi_{n,k,\sigma}(a,b')^{Y}=0$.  Call a formula $\phi_{n,k,\sigma}(a,y)$ \emph{admissible} if $(\inf_y \phi_{n,k,\sigma}(a,y))^{\B(H)}=0$.

Note that the set $W_{n,k,\sigma,a,\epsilon} := \{b\in X^k : \|b\|\leq 1,\ \phi_{n,k,\sigma}(a,b)^X<\e\}$ is a bounded subset of $X$, so its weak closure is weakly compact. Since $X$ is e.c., the family $(W_{n,k,\sigma,a,\epsilon})$, where we only consider admissible $\phi_{n,k,\sigma}(a,y)$, has the finite intersection property, whence the intersection of their weak closures is non-empty. This shows that for every $b\in \B(H)^k$, there is a u.c.p.\ map $\eta_b: X + \bC b_1 + \dotsb + \bC b_k\to \overline{X}$ which extends the identity on $X$. Letting $\mathcal F$ be the net all of finite subsets of self-adjoint elements of $\B(H)$ directed by inclusion, we have that any weak cluster point $\eta$ of $\{\eta_b : b\in\mathcal F\}$ is a u.c.p.\ map $\eta: \B(H)\to \overline{X}$ which extends the identity on $X$, whence $X$ is weakly injective.
\end{proof}

If $A$ is a C$^*$-algebra, then choosing the embedding $A\subset \B(H_u)$ in the universal representation, we have that that the following corollary holds by \cite[Corollary 6.3]{CE}.

\begin{cor}\label{ecwep} If $A$ is a unital C$^*$ algebra that is semi-existentially closed as an operator systems, then $A$ has the WEP.
\end{cor}

We would next like to offer an alternative proof of Corollary \ref{ecwep}.  First, for a positive element $a$ of an operator system $S$, we say that $a$ is \emph{strictly positive} if there is $r>0$ such that $a\geq r1$.  According to \cite[Section II.4.2]{Black}, if $A$ is a \emph{unital} C$^*$ algebra, then a positive element $a$ of $A$ is strictly positive if and only if $a$ is invertible.  It is an easy exercise to check that these two notions of strictly positive agree for elements of unital $C^*$ algebras.  We need the following recent theorem of Farenick, Kavruk, and Paulsen.

\begin{thm}[\cite{FKP}]
Suppose that $A\subseteq \B(H)$ is a unital $C^*$ algebra.  The following are equivalent:
\begin{enumerate}
\item $A$ has WEP;
\item Whenever, for arbitrary $p\in \n$, there exists $x_1,x_2\in M_p(A)$ and $a,b,c\in M_p(\B(H))$ such that $a+b+c=1$ and the matrix
$$\left[ \begin{matrix} a & x_1 & 0\\ x_1^* & b & x_2\\ 0 & x_2^* & c\end{matrix}\right]$$ is strictly positive in $M_{3p}(\B(H))$, then there also exist $\tilde a, \tilde b, \tilde c$ in $M_p(A)$ with the same property.
\end{enumerate}
\end{thm}

We can use the previous theorem to give an alternative proof of Corollary \ref{ecwep}.  We verify condition (2) in the previous theorem.  Fix $x_1,x_2,a,b,c$ as in the condition; call that matrix $u$.  Take $r>0$ such that $u-rI$ is positive in $M_{3p}(\B(H))$, say $u-rI=d^*d$.  Fix $s>0$ very small.  Then since $A$ is e.c., $A$ has elements $a',b',c',d'$ such that $a'+b'+c'$ is within $s$ of $1$ (in the max metric on $M_p(A)$) and such that $u'-r1$ is within $s$ of $(d')^*d'$ (again, with the max metric on $M_{3p}(A)$); here $u'$ denotes the corresponding matrix with the primed elements replacing the original elements.  By perturbing $a',b',c'$ we may assume that $a', b', c'$ are self-adjoint and add up to $1$, and perhaps now $u'-r1$ is within $2s$ of $(d')^*d'$.  Now by functional calculus, there is $t>0$ such that $u'-t1$ is positive, whence $u'$ is strictly positive again.

We would like to thank David Sherman for pointing out that the previous theorem ought to imply Corollary \ref{ecwep}.

\begin{rmk}
Corollary \ref{ecwep} yields yet another proof of the fact that e.c.\  exact C$^*$ algebras are nuclear.  Indeed, by the results of \cite{effhaag}, an exact C$^*$ algebra is nuclear if and only if it has WEP.
\end{rmk}



\begin{cor} If a unital inclusion $A\subset B$ of C$^*$-algebras is existential in the language of operator systems, then $A$ has the WEP if $B$ has the WEP.
\end{cor}

Since Kirchberg proved that the Connes Embedding Problem (CEP) is equivalent to C$^*(\mathbb F_\infty)$ having WEP, we see that CEP follows from C$^*(\mathbb F_\infty)$ being semi-existentially closed as an operator system.
(We already know that it is not e.c.\   as a C$^*$ algebra.)
In fact, since any injective von Neumann algebra has the WEP, CEP would follow as a consequence of there existing an embedding C$^*(\mathbb F_\infty)\hookrightarrow \mathcal R$ which is existential in the language of operator systems (again, the Dixmier property obstructs any such embedding from being existential in the language of C$^*$-algebras). Perhaps the existence of such an existential embedding (or at least such a \emph{positively} existential embedding, see below) is equivalent to CEP. Also, by \cite[Theorem 7.2]{FKP}, we have:

\begin{cor}
If $C^*(\mathbb F_2)$ is semi-existentially closed as an operator systems, then CEP holds.
\end{cor}

We now present a partial converse to Corollary \ref{ecwep}.  For the moment, we work in an arbitrary continuous signature $\la$.  Suppose that $\m$ and $\N$ are $\la$-structures.  We say a function $F:\m\to \N$ is a \emph{homomorphism} if it respects the interpretations of function and constant symbols and  is contractive with respect to predicate symbols, that is, $P^\N(F(a))\leq P^\m(a)$.  For example, u.c.p. maps between operator systems are homomorphisms.

We call a quantifier-free formula $\varphi(v)$ \emph{positive} if, for any $\la$-structures $\m$ and $\N$, any homomorphism $F:\m\to \N$, and any tuple $a$ from $M$, we have $\varphi(F(a))^\N\leq \varphi(a)^\m.$  This definition (which is a more general definition that that given in \cite[Section 3]{scowcroft} as the Lemma below establishes) is motivated by the fact from classical logic that a formula is positive (i.e., is built without using negations) if and only if its truth is preserved by homomorphisms. It turns out that many of the quantifier-free formulae considered above are positive.  In fact, we now give a general result indicating how to produce a large collection of positive formulae.


We say that a continuous function $f:[0,1]^n\to [0,1]$ is \emph{increasing} if $u_i\leq v_i$ for each $i$ implies $f(u_1,\ldots,u_n)\leq f(v_1,\ldots,v_n)$.  Notice that the $n$-ary connectives $\max$ and $\min$ as well as $n$-ary addition are all increasing.  The following lemma is proven by a routine induction on complexity of formulae:

\begin{lemma}
If a quantifier-free formula is built using only increasing connectives, then it is a positive formula.
\end{lemma}


We say that an $\la$-structure is \emph{positively existentially closed} (p.e.c.) if it satisfies the definition of being e.c.\ but only with respect to positive quantifier-free formulae.  (In \cite{scowcroft}, the author uses the term ``algebraically closed'' instead of p.e.c.)  There is also a natural notion of a C$^*$ algebra being \emph{semi-positively existentially closed} (semi-p.e.c.) as either an operator system or an operator space.

Inspecting the proof of Corollary \ref{ecwep}, we actually showed:

\begin{cor}\label{pecwep}
If $A$ is a unital C$^*$ algebra that is semi-p.e.c.\  as an operator systems, then $A$ has WEP.
\end{cor}

More generally, the proof of Proposition \ref{ecinj} shows:

\begin{cor}\label{pecrwi}
If $A\subseteq B$ are C$^*$ algebras such that $A$ is p.e.c.\ in $B$ as an operator system, then $B$ is weakly injective relative to $A$.
\end{cor}

Here is a partial converse to Corollary \ref{ecwep}.
\begin{prop}
Suppose that $A$ is a unital $C^*$ algebra with WEP.  Further suppose that, for every faithful representation $\pi:A\to \B(H)$, we have $A$ is positively e.c.\  in $\pi(A)''$ in the language of operator systems.  Then $A$ is positively e.c. as an operator system
\end{prop}

\begin{proof}
Suppose that $S$ is an operator system with $A\subseteq S\subseteq \B(H)$.  Using the conditional expectation $\B(H)\to \pi(A)''$ (where $\pi$ is the induced representation of $A$ into $\B(H)$) we see that any solution to a positive formula in $S$ can be mapped to one in $\pi(A)''$; now use the fact that $A$ is positively e.c.\   in $\pi(A)''$.
\end{proof}

Of course the preceding assumption can be replaced with the equivalent assumption that $A$ is positively e.c.\  in $A^{**}$.

Recall that a von Neumann algebra is injective if and only if it has WEP.  Consequently, we have:

\begin{cor}
For a von Neumann algebra $A$, we have that $A$ is injective if and only if it is positively e.c.\  in the class of operator systems.
\end{cor}

Examining the proof of Theorem \ref{exactec2}, we see that the quantifier-free formulae involved were positive.  We thus have:

\begin{prop}
Suppose that $A$ is a separable exact $C^*$ algebra.  Then $A$ is nuclear if and only if $A$ is positively e.c.\  in $\mathcal O_2$ in the language of operator systems.
\end{prop}

\begin{proof}
If $A$ is nuclear, then we have a conditional expectation from $\mathcal O_2$ to $A$, which implies that $A$ is positively e.c.\  in $\mathcal O_2$.  The remarks preceding this proposition yield the converse.
\end{proof}

We next remark on the connection between the WEP and existential closedness in the language of operator \emph{spaces}. For a map $\phi: A\to B$, let $\|\phi\|_n = \|\phi\otimes \id_{M_n(\bC)}\|$. Note that $\|\phi\|_n\leq \|\phi\|_{n+1}$ since $M_n(\bC)\to M_{n+1}(\bC)$ embeds (non-unitally). Consider the following properties:

\begin{itemize}

\item[($\alpha$)] For every unital inclusion $A\subset B$ of C$^*$-algebras and every finite-dimensional subspace $E\subset B$ and every $\e>0$ there exists a map $\phi: E\to A$ such that $\|\phi\|<1+\e$ and $\phi |_{E\cap A} = \id_{E\cap A}$.

\item[($\alpha'$)] For every unital inclusion $A\subset B$ of C$^*$-algebras and every finite-dimensional subspace $E\subset B$ and every $\e>0$ there exists a map $\phi: E\to A$ such that $\|\phi\|\leq 1$ and $\|\phi |_{E\cap A} - \id_{E\cap A}\|<\e$.

\item[($\beta$)] For every unital inclusion $A\subset B$ of C$^*$-algebras and every finite-dimensional subspace $E\subset B$ and every $\e>0$ there exists a map $\phi: E\to A$ such that $\|\phi\|_{\cb}<1+\e$ and $\phi |_{E\cap A} = \id_{E\cap A}$.

\item[($\beta'$)]  For every unital inclusion $A\subset B$ of C$^*$-algebras and every finite-dimensional subspace $E\subset B$ and every $\e>0$ there exists a map $\phi: E\to A$ such that $\|\phi\|_{\cb}\leq 1$ and $\|\phi |_{E\cap A} - \id_{E\cap A}\|<\e$.

\item[($\gamma$)] For every unital inclusion $A\subset B$ of C$^*$-algebras and every finite-dimensional subspace $E\subset B$ and every $n, \e>0$ there exists a map $\phi: E\to A$ such that $\|\phi\|_{n}<1+\e$ and $\phi |_{E\cap A} = \id_{E\cap A}$.

\item[($\gamma'$)] For every unital inclusion $A\subset B$ of C$^*$-algebras and every finite-dimensional subspace $E\subset B$ and every $n, \e>0$ there exists a map $\phi: E\to A$ such that $\|\phi\|_{n}\leq 1$ and $\|\phi |_{E\cap A} - \id_{E\cap A}\|<\e$.

\end{itemize}

Clearly we have that $(\beta) \Rightarrow (\gamma) \Rightarrow (\alpha)$. Note that $(\beta) \Leftrightarrow (\beta')$ by standard operator space perturbation techniques, and the equivalences $(\alpha) \Leftrightarrow (\alpha')$ and $(\gamma) \Leftrightarrow (\gamma')$ are standard. Then $A$ has the WEP $\Leftrightarrow (\alpha')$ (the non-trivial implication is the principle of local reflexivity for Banach spaces, see Lemma 13.3.2 in \cite{BO} for a proof of this equivalence), and $(\beta) \Rightarrow A$ is semi-p.e.c.\ as an operator space (c.b.\ maps are homomorphisms in the category of operator spaces). In fact, it is not much harder to see that $(\gamma)$ is sufficient to ensure $A$ is semi-p.e.c.\ as an operator space.

\begin{prop} If $A$ has the WEP, then $A$ is semi-p.e.c.\ as an operator space.
\end{prop}

\begin{proof} We need only show that WEP $\Rightarrow (\gamma')$. Let $\phi: B\to A^{**}$ be a weak conditional expectation.  Since $\phi$ is u.c.p., $\phi_n: M_n(B) \to M_n(A^{**})\cong M_n(A)^{**}$ is contractive; since $\phi_n$ restricts to the identity on $M_n(A)$, it follows that $\phi_n$ is a weak conditional expectation.


Let $F\subset B$ be a finite-dimensional operator system and define $E := F\cap A$ so that $M_n(E) = M_n(F)\cap M_n(A)$. By the princple of local reflexivity, we have that for every $\e>0$, there is a contraction $\phi': M_n(F)\to M_n(A)$ so that $\|\id_{M_n(E)} - \phi' |_{M_n(E)}\|<\e$. Since interchanging rows or columns is an isometry on $M_n(B)$ or $M_n(A)$ and $\id_{M_n(E)}$ commutes with these operations, by averaging we may assume without loss of generality that $\phi'$ commutes with interchanging any rows or columns. In particular $\phi'(e_{ij}\otimes x)_{ij} = \phi''(e_{kl}\otimes x)_{kl}$ for any $i,j,k,l\in \{1,\dotsc,n\}$, $x\in F$. It is easy to see that $\phi''(e_{ij}\otimes x) := e_i \phi'(e_{ij}\otimes x) e_j$ is again a contraction so that $\phi'' = \id_{M_n(\bC)}\otimes\psi$ with $\|\id_E - \psi |_{E}\|<\e$.
\end{proof}

We have thus seen that, for a unital C$^*$ algebra $A$, we have that the following conditions are successively weaker:

\begin{enumerate}
\item $A$ is semi-p.e.c.\  as an operator system;
\item $A$ has WEP;
\item $A$ is semi-p.e.c.\  as an operator space.
\end{enumerate}

This leads naturally to the following

\begin{question}\label{coincide}
For a unital C$^*$ algebra $A$, do we have that $A$ is semi-p.e.c.\  as an operator system if and only if $A$ is semi-p.e.c.\  as an operator space?
\end{question}

A positive answer to the above question would have two nice consequences.  First, there would now be two new nontrivial definitions of WEP for operator systems.  Secondly, we would have a new equivalent formulation of CEP, namely that C$^*(\mathbb F_\infty)$ is semi-p.e.c.\ as an operator system.

Recall from \cite{effhaag} that a C$^*$ algebra $A$ is said to be \emph{approximately injective} if given finite-dimensional operator systems $E_1\subseteq E_2\subseteq \B(H)$, a completely positive map $\phi_1:E_1\to A$, and $\epsilon>0$, there is a completely positive map $\phi_2:E_2\to A$ such that $\|\phi_2|_{E_1}-\phi_1\|<\epsilon$.  It is shown in \cite[Proposition 4.5]{effhaag} that approximate injectivity implies WEP.  We leave the following proposition as an easy exercise.

\begin{prop}\label{appinjpec}
If $A$ is an approximately injective C$^*$ algebra, then $A$ is p.e.c.\  as an operator system. In particular, being approximately injective, being (semi-)p.e.c.\ as an operator system, and having WEP are successively weaker conditions.
\end{prop}

\begin{question}\label{coincide2}
For a unital C$^*$ algebra $A$, can either of the previous implications be reversed? That is, do we have either: (A) $A$ is p.e.c.\  as an operator system if and only if $A$ is approximately injective; or (B) $A$ is p.e.c.\  as an operator system if and only if $A$ has WEP?
\end{question}

Note that we cannot have positive answers to both Questions \ref{coincide2}(A) and (B), for otherwise we would have that the notions of WEP and approximately injective coincide for separable C$^*$ algebras, contradicting Corollary 3.1 of \cite{jungepisier}.  Corollary 3.1 of \cite{jungepisier} also implies that C$^*(\mathbb F_\infty)$ is not approximately injective, whence a positive answer to Question \ref{coincide2}(A) would imply that C$^*(\mathbb F_\infty)$ is not p.e.c.\ as an operator system. Further, a negative answer to (B) would imply a negative answer to Question \ref{coincide}.


We have seen that the difference between the WEP and approximate injectivity lies in the fact that with the WEP, given $E_1\subset E_2$ finite-dimensional operator systems, and a u.c.p.\ map $\phi: E_1\to A$, we are only able to find for any $n$ an $n$-contractive approximate extension of $\phi$ to $E_2$. With this in mind, we make the following definition:

\begin{df}\label{cp-stable} An operator system $X$ is said to be \emph{CP-stable} if for any for any finite-dimensional subspace $E_1\subset X$ and $\delta$ there exists a finite-dimensional subspace $E_2\supset E_1$ and $n,\e$ so that for any unital map $\phi: E_2\to A$ into any unital C$^*$-algebra $A$ with $\|\phi\|_n< 1+ \e$, there exists a u.c.p.\ map $\psi: E_1\to A$ so that $\|\phi |_{E_1} - \psi\|<\delta$. If, in addition, we can ensure that $\|\phi |_{E_1}-\psi\|_{\cb}<\delta$, then we say that $X$ is \emph{strongly CP-stable}.  In this case, note that $\phi |_{E_1}$ is completely bounded with $\|\phi |_{E_1}\|_{\cb}< 1+ \delta$.
\end{df}

\begin{prop}\label{cp-stable-wep} Let $X$ be an operator system which is contained in a CP-stable operator system and let $A$ be a C$^*$-algebra with the WEP. Then for any finite-dimensional subsystems $E_1\subset E_2\subset X$, any u.c.p.\ map $\phi: E_1\to A$, and every $\e>0$ there exists a u.c.p.\ map $\phi': E_2\to A$ so that $\|\phi' |_{E_1} - \phi\|<\e$.
\end{prop}

\begin{proof} Suppose that $X\subset \widetilde X$ with $\widetilde X$ CP-stable. Let $E_1\subset E_2\subset X$ be finite-dimensional operator systems, $\phi: E_1\to A$ a u.c.p.\ map, and $\delta>0$. By CP-stability we can then find $E_2\subset E_3\subset \widetilde X$ finite-dimensional and $n,\e$ so that for any unital map $\phi': E_3\to A$ so that $\|\phi'\|_n<1+\e$,  there is a u.c.p.\ map $\phi'': E_2\to A$ so that $\|\phi' |_{E_2} - \phi''\|< \delta$. Since $A$ has the WEP, the set of all such maps $\phi': E_3\to A$ is non-empty; further, we may find such $\phi'$ with $\|\phi' |_{E_1} - \phi\|<\delta$. Therefore $\|\phi'' |_{E_1} - \phi\|<2\delta$, and we are done.

\end{proof}

The following result is probably well known but seems hard to source in the literature. In particular it demonstrates that $M_k$ is strongly CP-stable. (We thank Martino Lupini for showing us how Choi's theorem could be used to eliminate one of the parameters in the conclusion.)  Before stating this result we introduce some notation. For any linear map $\phi: M_k\to A$, we define $\hat\phi:=[\phi(e_{ij})]_{ij}\in M_k(A)$.

\begin{prop} Fixing $k$, for any $\delta>0$ there exists $\e>0$ so that for any unital map $\phi: M_k\to A$, $A$ an arbitrary unital C$^*$-algebra, with $\|\phi\|_k<1+\e$, there exists a u.c.p.\ map $\psi: M_k\to A$ so that $\|\phi -\psi\|_{\cb}\leq 1+\delta$.
\end{prop}

\begin{proof} We claim that for any $\eta$, there exists $\e$ so that for $\phi$ as above there exists $a\in M_k(A)^+$, $\|a\|\leq1$ so that $\|\hat\phi - a\|<\eta$.  Given the claim, we now show how the proposition follows. If $\psi_a: M_k\to A$ is the c.c.p.\ map induced by $a$, then by the ``small perturbation argument'' (see Lemma 12.3.15 in \cite{BO}), choosing the dual basis $\{e_{ij}^*\}$, we have that $\|\phi - \psi_a\|_{cb}<k^2\eta =: \delta$. Since $\|\psi_a(1) - 1\|<\delta$, if $\delta<1$, $b := \psi_a(1)$ is invertible, whence setting $\psi(x) := b^{-1/2}\psi_a(x)b^{-1/2}$ yields the desired u.c.p.\ map.

We now prove the claim.  Suppose, towards a contradiction, that the claim is false, that is, there exists $\eta>0$ so that for any $\e>0$ there exists a unital map $\phi: M_k\to A$ in to some C$^*$-algebra (we may assume separable) with $\|\phi\|_k<1+\e$ so that ${\rm dist}(\hat\phi, M_k(A)^+_{\leq 1})\geq \eta$. For $j\geq 1$, let $\phi_j: M_k\to A_j$ be such a map corresponding to $\epsilon=\frac{1}{j}$. Consider the ultrapower map $\phi:=(\phi_j)^\bullet: M_k\to \prod_\omega A_j =: \mathcal A$. Then it is easy to see that $\phi$ is unital and $\|\phi\|_{k}=1$, whence $\phi$ is $k$-positive and hence u.c.p. by Choi's Theorem (see, for example, \cite[Theorem 3.1.4]{Paulsen}). This implies that $\hat\phi = (\hat\phi_j)^\bullet$ is a positive contraction in $M_k(\mathcal A)$; however, this cannot be the case since $M_k(\mathcal A) \cong \prod_\omega M_k(A_j)$ and each $\hat\phi_j$ is $\eta$-separated form the positive contractions in $M_k(A_j)$.
\end{proof}

We remark that setting $X=M_k$ in Proposition \ref{cp-stable-wep} recovers an observation of Effros and Haagerup that WEP implies ``finite approximate injectivity.'' (See the remarks following Proposition 4.5 in \cite{effhaag}.) Note our argument differs slightly from the one given there.

\begin{cor} The class of strongly CP-stable operator systems contains all nuclear operator systems.
\end{cor}

%

We end this section on the connection between existential closedness and the \emph{lifting property}.  Recall that a unital C$^*$ algebra $A$ has the lifting property (LP) if, for every unital C$^*$ algebra $C$, every closed ideal $J$ of $C$, and every u.c.p.\  map $\rho:A\to C/J$, there is a u.c.p.\ lift $\tilde{\rho}:A\to C$ (i.e., $\pi_J\circ \tilde{\rho}=\rho$, where $\pi_j:C\to C/J$ is the canonical projection map).  If, in the preceding definition, we can only ensure that $\rho|X$ has a u.c.p.\ lift for every finite-dimensional operator system $X\subseteq A$, then we say that $A$ has the \emph{local lifting property} (LLP) of Kirchberg.

We now show that, unlike WEP, not every e.c.\ C$^*$ algebra has the LLP.  We need a preparatory result giving an alternate characterization of CP-stability for C$^*$ algebras in terms of an appropriate lifting property:

\begin{prop}
Suppose that $A$ is a separable unital C$^*$ algebra.  Then $A$ is CP-stable if and only if $A$ has the \emph{local ultrapower lifting property} (LULP), namely, for every unital C$^*$ algebra $B$, every nonprincipal ultrafilter $\omega$ on $\n$, every u.c.p.\ map $\phi:A\to B^\omega$, and every finite-dimensional operator system $E\subseteq A$, there are u.c.p.\ maps $\phi_n:E\to B$ such that $\phi |_E=(\phi_n)^\bullet$.  In particular, every separable C$^*$ algebra with the LLP is CP-stable.
\end{prop}

\begin{proof}
First assume that $A$ is CP-stable and fix a u.c.p.\  map $\phi:A\to B^\omega$.  Fix a finite-dimensional operator system $E\subseteq A$.  Fix $n\geq 1$.  By CP-stability, there are finite-dimensional $E'\supseteq E$ and $m$, $\e$ such that, for any unital map $\psi: E'\to C$ into any unital C$^*$-algebra $C$ with $\|\psi\|_m< 1+ \e$, there exists a u.c.p.\ map $\theta: E\to C$ so that $\|\psi |_E - \theta\|<\frac{1}{2n}$.  Write $\phi |_{E'}=(\psi_k)^\bullet$, where each $\psi_k:E\to B$ is a unital linear map.  Since $\phi |_{E'}$ is u.c.p., for almost all $k$ we have that $\|\psi_k\|_m<1+\e$.  It follows that, for almost all $k$, there is u.c.p.\ map $\theta:E\to B$ such that $\|\psi_k |_E-\theta\|<\frac{1}{2n}$.  Also, for almost all $k$, we have $\|\psi_k |_E-\phi |_E\|<\frac{2}{n}$.  Thus, there is a u.c.p.\ map $\phi_n:E\to B$ such that $\|\phi_n-\phi |_E\|<\frac{1}{n}$.  It follows that $\phi |_E=(\phi_n)^\bullet$.

Conversely, suppose that $A$ has the LULP and yet $A$ is not CP-stable as witnessed by some finite-dimensional $E\subseteq X$ and some $\delta>0$.  Let $$E=E_0\subseteq E_1\subseteq E_2\subseteq \cdots$$ be an increasing sequence of finite-dimensional operator subsystems of $A$ whose union is dense in $A$.  For each $n\geq 1$, we have a unital linear map $\phi_n:E_n\to \B(H)$ such that $\|\phi_n\|_n<1+\frac{1}{n}$ and yet $\phi_n |E$ is not within $\delta$ of a u.c.p.\ map $E\to \B(H)$.  Note that $A$ is a subalgebra of $\prod_\omega E_n$, so we may define $\phi:A\to \B(H)^\omega$ by $\phi=(\prod_\omega \phi_n) |_A$.  Since $\|\phi_n\|_n<1+\frac{1}{n}$ for all $n$, it follows that $\phi$ is u.c.p.  By the LULP, there are u.c.p.\ $\psi_n:E\to \B(H)$ such that $\phi |_E=(\psi_n)^\bullet$.  For almost all $n$, we have that $\|\phi_n |E-\psi_n\|<\delta$, contradicting the fact that each $\phi_n$ is at least $\delta$ away from any u.c.p.\ map $E\to \B(H)$.
\end{proof}

\begin{cor}
There exists a separable C$^*$ algebra that is not contained in any CP-stable C$^*$ algebra.  In particular, it is not the case that every separable e.c.\ C$^*$ algebra is CP-stable, whence it follows that not every separable e.c.\ C$^*$ algebra has the LLP.
\end{cor}

\begin{proof}
Suppose that every separable C$^*$ algebra is contained in a CP-stable C$^*$ algebra.  Then by Proposition  \ref{cp-stable-wep}, we have that every C$^*$ algebra with WEP is approximately injective, contradicting Corollary 3.1 of \cite{jungepisier}.  (See Remark (iv) after the aforementioned corollary.)
\end{proof}

\begin{cor}
Assuming KEP, $\O_2^\omega$ does not have the LLP, whence having the LLP is not an elementary property.
\end{cor}

\begin{proof}
Suppose, towards a contradiction, that $\O_2^\omega$ had LLP.  Let $A$ be a separable e.c.\ C$^*$ algebra.  Then $\O_2^\omega$ is weakly injective relative to $A$ by Lemma \ref{pecrwi}, whence we would have that $A$ has LLP by \cite[Corollary 2.6(v)]{K}.  Since $A$ was an arbitrary separable e.c.\ C$^*$ algebra, we get a contradiction to the previous corollary.
\end{proof}

\begin{rmk}
If KEP holds, then a C$^*$ algebra $A$ is semi-p.e.c.\ as an operator system if and only if it is p.e.c.\ in $\O_2^\omega$, in which case $A$ has WEP.  If one drops the KEP assumption but still assumes that $A$ is p.e.c.\ in $\O_2^\omega$, then one can still conclude that $A$ is WEP \emph{if} one also assumes that $A$ has the LLP.  Indeed, if $A$ is p.e.c.\ in $\O_2^\omega$, then $\O_2^\omega$ is weakly injective relative to $A$ by Lemma \ref{pecrwi}.  By \cite[Corollary 3.3(i)]{K}, $\ell^\infty(\O_2)$ has WEP, whence $\O_2^\omega$ is QWEP (quotient of a C$^*$ algebra with WEP).  By \cite[Corollary 3.3(iii)]{K}, it follows that $A$ is QWEP.  Finally, by \cite[Corollary 3.6(ii)]{K}, we have that $A$ is WEP.
\end{rmk}

\section{Kirchberg's embedding problem}

\subsection{General remarks}
Kirchberg's embedding problem (KEP) asks whether every separable C$^*$ algebra embeds into an ultrapower of $\O_2$.  As stated this is a bit ambiguous as it's not clear whether we are working in the unital or nonunital category.  It turns out that the unital and nonunital versions of KEP are equivalent.  Indeed, if every separable, unital C$^*$ algebra embeds unitally into $\O_2^\omega$, then by considering unitizations, we see that every separable C$^*$ algebra embeds into $\O_2^\omega$.  Conversely, suppose that every separable C$^*$ algebra embeds into $\O_2^\omega$.  Suppose that $A$ is a separable, unital C$^*$ algebra and $f:A\to \O_2^\omega$ is a (not necessarily unital) embedding.  Let $p:=f(1)$, a projection in $\O_2^\omega$.  Notice then that $f:A\to p\O_2^\omega p$ is a unital embedding.  We may write $p=(p_n)^\bullet$, with each $p_n$ a projection in $\O_2$.  Since $p_n\O_2 p_n\cong \O_2$ (unitally) for each $n$, we see that $p\O_2^\omega p \cong \O_2^\omega$ (unitally).


 In some sense, KEP is the C$^*$ analog of the Connes Embedding Problem.  We now want to point out an interesting difference between the two problems.  It is known that if $M$ is a (separable) II$_1$ factor that is $\R^\omega$-embeddable and $\alpha$ is an automorphism of $M$, then $M\rtimes_\alpha \z$ is once again $\R^\omega$-embeddable.  It turns out that the analogous statement for $\O_2^\omega$-embeddability is equivalent to KEP.

\begin{df} Let $\cS$ denote the universal, separable UHF algebra. A C$^*$-algebra $A$ is said to be \emph{quasi-diagonal} if there exists an embedding $A\to \cS^\omega$ which admits a c.c.p. lift (u.c.p.\ if $A$ is unital).
\end{df}

This is not the usual definition, but is easily seen to be equivalent. Since $\cS$ embeds into $\O_2$, we see that quasi-diagonal C$^*$ algebras are $\O_2^\omega$-embeddable.

For any C$^*$-algebra $A$ consider the C$^*$-algebra $C_0(\r)\otimes A$, which is the same as the algebra $C_0(\r,A)$ of all continuous functions $f: \r \to A$ for which $\lim_{x\to \pm\infty} f(x) = 0$.  Voiculescu proved that $C_0(\r,A)$ and its unitization $B$ are always quasi-diagonal (see \cite[Theorem 5]{V}).

We now follow the argument given in \cite[Theorem 6.3.11]{Ror}.  We let $\alpha$ be the automorphism of $C_0(\r,A)$ given by $\alpha(f(t)) = f(t+1)$ and extend $\alpha$ to an automorphism (also denoted $\alpha$) of $B$.  We then have that $\k\otimes A\cong C_0(\r,A)\rtimes_\alpha \z \subseteq B\rtimes_\alpha \z$.  Thus, if $\O_2^\omega$-embeddable (even quasidiagonal) algebras remain $\O_2^\omega$-embeddable after taking a cross product with $\z$, then every separable C$^*$ algebra embeds into $\O_2^\omega$.  We just proved:

\begin{prop}
KEP is equivalent to the statement that the class of $\O_2^\omega$-embeddable algebras is closed under crossed products by $\z$.
\end{prop}

As in the case of CEP, KEP is equivalent to the statement that $\Th_\exists(\O_2)\leq \Th_\exists(A)$ for every C$^*$ algebra $A$.  The following model-theoretic equivalents of KEP are somewhat deeper.

\begin{thm}\label{KEPequiv}
The following are equivalent:
\begin{enumerate}
\item KEP;
\item $\mathcal O_2$ is existentially closed;
\item there is an exact e.c.\   $C^*$ algebra.
\end{enumerate}
\end{thm}

\begin{proof}
By Lemma \ref{ssa}, $\O_2$ is an e.c.\   model of its universal theory; KEP implies that every (unital) C$^*$ algebra is a model of the universal theory of $\O_2$, whence $\O_2$ is actually e.c.\ in the class of all (unital) C$^*$ algebras.  For the converse, note that e.c.\   $C^*$ algebras have the same existential (even $\forall\exists$) theories, so if $A$ is a separable $C^*$ algebra embedded into the e.c.\   $C^*$ algebra $B$, we have
$$\Th_\exists(\mathcal O_2)= \Th_\exists(B)\leq \Th_\exists(A),$$ whence $A$ embeds into an ultrapower of $\mathcal O_2$.  (2) implies (3) is trivial and the converse follows from  Theorem \ref{exactec} and Proposition \ref{onlypossibility}.
\end{proof}

\begin{rmk}
The preceding theorem yields yet another reformulation of KEP, namely that KEP is equivalent to the assertion that every separable nuclear C$^*$ algebra is contained in a separable e.c.\ nuclear C$^*$ algebra.  Moreover, KEP is also equivalent to the assertion that every separable exact C$^*$ algebra is contained in a separable e.c.\ exact C$^*$ algebra.
\end{rmk}

\begin{rmk} Suppose $A$ is a C$^*$-algebra which is an e.c.\ model of its universal theory. Let KEP$_A$ be the statement, ``every separable C$^*$-algebra embeds in an ultrapower of $A$.'' Then as in the previous theorem KEP$_A$ is equivalent to $A$ being existentially closed. In particular, if $A$ is strongly self absorbing, then it follows easily from Proposition \ref{flip} that KEP$_A$ is known to be false unless $A\cong \O_2$. Is there an operator algebraic proof of this fact (even just for the case $A = \O_\infty$)?
\end{rmk}

\subsection{Good nuclear witnesses}

Following \cite{FHRTW}, for a C$^*$ algebra $A$ and an $m$-tuple $a$ from $A$, define $$\Delta^A_{\nuc m}(a):=\inf_{F,\phi,\psi}\max_{i\leq m} \|(\psi\circ \phi)(a)-a\|,$$ where $F$ ranges over all finite-dimensional C$^*$ algebras, $\phi$ ranges over all cpc maps $\phi:A\to F$ and $\psi$ ranges over all cpc maps $\psi:F\to A$.  So $A$ is nuclear if and only if $\Delta^A_{\nuc m}(a)=0$ for all $m$ and all $m$-tuples $a$ from $A$.  Note that, if $A\subseteq B$ and $a$ is an $m$-tuple from $A$, we have $\Delta^B_{\nuc m}(a)\leq \Delta^A_{\nuc m}(a)$ as every cpc map $\phi:A\to F$ can be extended to one $B\to F$.

%

\begin{fact}[\cite{FHRTW}]
For each $m\geq 1$, there are \emph{existential} formulae $\Phi_{m,n}(x)$, with $x$ and $m$-tuple of variables, such that, for any C$^*$ algebra $A$ and any $m$-tuple $a$ from $A$, we have
$$\Delta^A_{\nuc m}(a)=\inf_n \Phi_{m,n}(a).$$
\end{fact}

\begin{rmk}
The above fact gives a simpler (albeit less elementary) proof of the fact that if $A$ and $B$ are simple unital C$^*$ algebras and $A$ is e.c.\   in $B$, then $B$ nuclear implies $A$ nuclear.  Indeed, suppose that $A$ is not nuclear.  Then there is some $m$ and some $k\geq 1$ such that the partial type $\{\Phi_{m,n}(x)\geq \frac{1}{k} \ : \ n\geq 1\}$ is realized in $A$, say by $a\in A$.  Since each $\Phi_{m,n}(x)$ is existential and since $A$ is e.c.\   in $B$, we have that $\Phi_{m,n}^B(a)\geq \frac{1}{k}$ for each $n\geq 1$, whence $B$ is not nuclear.
\end{rmk}

Let us define a \emph{condition} to be a finite set $p$ of expressions of the form $\varphi(x)<r$, where $\varphi(x)$ is quantifier-free and $r$ is a positive real number.  We say that a tuple $a$ from a C$^*$ algebra $A$ \emph{satisfies the condition $p$} if $\varphi(a)^A<r$ for every condition $\varphi(x)<r$ in $p$.  Note that a condition is satisfiable if and only if it is satisfied in $\B(H)$ for separable $H$.

In order to motivate the main result of this section, let us first note the somewhat easy equivalent reformulation of KEP:

\begin{prop}
The following are equivalent:
\begin{enumerate}
\item KEP;
\item every satisfiable condition is satisfied in a nuclear C$^*$ algebra;
\item every satisfiable condition is satisfied in an exact C$^*$ algebra.
\end{enumerate}
\end{prop}

\begin{proof}
First suppose that KEP holds and that $p(x)$ is a condition.  Suppose that $a$ is a tuple from a C$^*$ algebra $A$ satisfying $p$.  Without loss of generality, we may suppose that $A$ is separable.  Suppose $p(x)=\{\varphi_i(x)<r_i \ : \ i=1,\ldots,m\}$.  Let $s_i:=\varphi_i(a)<r_i$.  Let $\sigma(x):=\max_{1\leq i\leq m}(\varphi_i(x)\dotminus s_i)$.  Then since $A$ is embeddable in $\O_2^\omega$, we have $(\inf_x \sigma(x))^{\O_2}=(\inf_x \sigma(x))^{\O_2^\omega}=0$.  Thus, if $\epsilon>0$ is small enough such that $s_i+\epsilon<r_i$ for each $i$, then by choosing $b\in \O_2$ such that $\sigma(b)<\epsilon$, we have that $b$ satisfies $p(x)$ and (2) holds.

Clearly (2) implies (3).  Finally, suppose that (3) holds.  We want $\Th_\exists(\O_2)\leq \Th_\exists(A)$ for every (separable) C$^*$ algebra $A$.  Suppose that $\sigma:=\inf_x\varphi(x)$ is an existential sentence and $r:=\sigma^A$.  Fix $\epsilon>0$ and consider the condition $p(x):=\{\varphi(x)<r+\epsilon\}$.  Since $p(x)$ is realized in $A$, there is an exact C$^*$ algebra $B$ such that $p(x)$ is realized in $B$, say by $a\in B$.  Let $B_1$ be the subalgebra of $B$ generated by $a$.  Then $B_1$ is separable, exact, so embeds in $\O_2$ by Kirchberg's Exact Embedding Theorem (see \cite[Theorem 6.3.11]{Ror}).  It follows that $p(x)$ is realized in $\O_2$, so $\sigma^{\O_2}\leq r+\epsilon$.  Since $\epsilon$ was arbitrary, we get that $\sigma^{\O_2}\leq r$.
\end{proof}

The idea is to now weaken the hypothesis in (2).

\begin{df}\label{gnw} Let us say that a satisfiable condition $p(x)$ has \emph{good nuclear witnesses} if, for every $\epsilon>0$, there is a C$^*$ algebra $A$ and a tuple $a$ from $A$ that realizes $p(x)$ and such that $\Delta^A_{\nuc}(a)<\epsilon$.

\end{df}

\begin{thm}\label{witnesses}
The following are equivalent:
\begin{enumerate}
\item KEP;
\item every satisfiable condition has good nuclear witnesses.
\end{enumerate}
\end{thm}

\begin{proof}
(1) implies (2) follows from the previous theorem.  We now show (2) implies (1).  The idea is to use a version of the Omitting Types Theorem for continuous logic.  The version we use is Corollary 4.7 from \cite{BI}; for the sake of the reader, we include the formulation of this result in Appendix A.

Let $\varphi_{m,k}$ be the infinitary sentence $\sup_x\Delta_{\nuc m}(x)$.  (Yes, we know that $k$ doesn't appear in $\varphi_{m,k}$.)  Let $r_{m_k}:=\frac{1}{k}$.  Then assumption (2) implies that $p\cup \{\Delta_{\nuc m}(x)<\frac{1}{k}\}$ is satisfiable for any satisfiable condition $p$.  Thus the Omitting Types theorem applies and gives us a canonical structure $M$ (also defined in Appendix A) such that $\varphi_{m,k}^M\leq \frac{1}{k}$ for every $m,k$, i.e. $\sup_x \Delta^M_{\nuc m}(x)=0$ for all $m$.

The proof of the Omitting Types Theorem in \cite{BI} shows that the aforementioned canonical structure $M$ is a so-called \emph{finitely generic} structure.  In particular, $M$ is existentially closed!  (This is well-known in classical logic; see \cite[Theorem 8.13]{Keisler} for a proof.  We give a proof of this in the continuous case in Appendix A.)  It follows that $M$ is existentially closed and nuclear, whence KEP holds (and in fact $M$ must be $\O_2$!).

\end{proof}

At this point two remarks are in order.  First, define a \emph{basic condition} to be a finite set of formulae of the form $|\|t(x)\|-r|<\epsilon$ for some *polynomial $t(x)$ and some $r$ and $\epsilon$.  We claim that in the above proof it is enough to ask that satisfiable basic conditions have good nuclear witnesses.  Indeed, suppose that satisfiable basic conditions have good nuclear witnesses and that $p$ is a satisfiable condition.  Let $a$ be a tuple satisfying $p$ in $\B(H)$.  The quantifier-free formulae in $p$ are of the form $f(\varphi_1,\ldots,\varphi_n)$, where each $\varphi_i$ of the form $\|t(x)\|$ for some term $t$.  Let $t_1,\ldots,t_m$ enumerate all terms appearing in the q.f. formulae in $p$ and let $r_i:=\|t_i(a)\|$.  Then for each $\delta>0$, by assumption the satisfiable basic condition asking that $|\|t_i(x)\|-r_i|<\delta$ has good nuclear witnesses.  If we choose $\delta$ small enough, such witnesses also witness $p$.

Second, let us remark that good nuclear witnesses can be reformulated in an \emph{a priori} weaker way, namely that, for every satisfiable condition $p$ and $\epsilon>0$, there is a tuple $a$ in $\B(H)$ witnessing $p$ and such that there are cpc maps $\phi:S\to M_k$ and $\psi:M_k\to \B(H)$ for which $\|(\psi\circ \phi)(a)-a\|<\epsilon$, where $S$ is the operator system generated by $a$.  Indeed, if this is the case, let $A$ be the C$^*$ algebra generated by $a$ and the image of $\psi$.  We can then extend $\phi$ to a cpc map $\phi:A\to M_k$ and it remains to notice that $\psi$ takes values in $A$.

\

It is useful to define a local version of good nuclear witnesses. For a tuple $a$ in a C$^*$-algebra, we define the ``quantifier-free type'' of $a$ to be the subset of all conditions $p(x)$ for which it holds that $\varphi(a)^A<r$ for all $``\varphi(x)<r"\in p$. We say that the tuple $a$ has \emph{good nuclear witnesses} if for every $\e>0$ and every condition $p(x)$ in the quantifier-free type of $a$, there is a tuple $a'\in \B(H)$ which satifies $p$ with $\Delta_{\nuc}^{\B(H)}(a')<\e$. In other words, there exists a sequence $(a_n)$ of tuples in $\B(H)$ so that $\Delta_{\nuc}^{\B(H)}(a_n)<1/n$ and so that the obvious map $a\mapsto (a_i)^\bullet$ extends to an embedding of $C^*(a)$ into $\B(H)^\omega$.
Naturally, we will say that a C$^*$-algebra $A$ admits good nuclear witnesses if every finite tuple $a\in A$ does. By the previous remark, every satisfiable condition has good nuclear witnesses if and only if every (separable) C$^*$-algebra has good nuclear witnesses. Additionally, it follows by the same argument as in Theorem \ref{witnesses} that $A$ admits good nuclear witnesses if there is a nuclear C$^*$-algebra which is an e.c.\ model of the universal theory of $A$.

We can see that having good nuclear witnesses is an \emph{a priori} weaker property than $\O_2^\omega$-embeddability, though it shares many of the key properties, e.g., being closed under substructure, tensor products, and (full) free products. A very natural question to ask then is whether the local version of Theorem \ref{witnesses} holds:

\begin{question} Does a separable C$^*$-algebra $A$ have good nuclear witnesses if and only if it is $\O_2^\omega$-embeddable?
\end{question}

For an exact $C^*$-algebra $A$ there is a well-developed notion of entropy for an action on $A$ by an automorphism due to Voiculescu and Brown (see \cite{Brown}).

\begin{question} Can one extend this to an entropy theory for actions by automorphisms for the class of $C^*$-algebras which admit good nuclear witnesses?
\end{question}

\noindent Since KEP is equivalent to permanence of the class of $\O_2^\omega$-embeddable C$^*$-algebras under crossed products by general automorphisms, such an entropy theory could provide important insight into this problem.

We now remark on the connection between good nuclear witnesses and the phenomena of soficity/hyperlinearity for discrete groups. See \cite{pestov} for an introductory treatment of this topic. To recall, a discrete group is said to be \emph{sofic}
if it is embeddable in a metric ultraproduct of finite symmetric groups equipped with the Hamming distance. The closest known analog of soficity in the category of C$^*$-algebras is the embeddability of $C_r^*(G)$ in a C$^*$-ultraproduct of matrix algebras (in which case $C_r^*(G)$ admits good nuclear witnesses). In contrast with soficity few groups are known to satisfy the latter, though it is an open problem whether these two properties are equivalent.

From such observations it is reasonable to draw a loose analogy between KEP and the open problem of whether every discrete group is sofic. We venture to make this slightly more precise.

\begin{question}  Does a stably finite C$^*$-algebra have good nuclear witnesses if and only if it embeds in an ultraproduct of matrix algebras? For a discrete group $G$ are either of these conditions on $C_r^*(G)$ equivalent to soficity?
\end{question}

As very preliminary evidence, we note that if $G$ embeds in an ultraproduct $\prod_\omega H_n$ of discrete amenable groups (whence $G$ is sofic), then the faithful unitary representation $\pi$ of $G$ induced by the embedding extends to an embedding $C^*(\pi(G))\to \prod_\omega C^*(H_n)$, whence $C^*(\pi(G))$ admits good nuclear witnesses. In fact it is easy to see that $\pi$ weakly contains the left-regular representation, so $C_r^*(G)$ is  quotient of $C^*(\pi(G))$.

\

Finally, we remark on the formulation of an operator system version of KEP. It is not hard to see that KEP implies that every separable C$^*$-algebra $A$ (equivalently, every separable operator system) has a complete order embedding as an operator system in $\cS^\omega$, i.e., there is a u.c.p.\ embedding $\phi: A\to \cS^\omega$ so that $\phi^{-1}: \phi(A)\to A$ is also u.c.p.  (Recall that $\cS$ is the universal, separable UHF C$^*$ algebra.) However, it is easy to see that this fact holds without assuming KEP.  Indeed, fix a separable operator system $E$ realized concretely in $\B(H)$ for a separable Hilbert space $H$.  Choose an increasing sequence  $(p_n)$ of finite rank projections in $H$ converging strongly to the identity.  Then the map $x\mapsto (p_n x p_n)^\bullet:E\to \prod_\omega M_{r(n)}(\bC)\subset \cS^\u$ is a complete order embedding, where $r(n) := {\rm rank}(p_n)$.

In \cite{MTOA3}, the authors call a model $\m$ of some theory $T$ \emph{locally universal} if every model of $T$ embeds in an ultrapower of $\m$; equivalently, if $\Th_\forall(\cN)$ is dominated by $\Th_\forall(\m)$ for every $\cN\models T$.  Thus CEP (resp., KEP) asks whether or not $\cR$ is locally universal for the theory of tracial von Neumann algebras (resp., whether or not $\O_2$ is locally universal for the theory of C$^*$ algebras).  By abstract model theory, one can show that there is a locally universal object in each of the classes of tracial von Neumann algebras, C$^*$ algebras, and operator systems.  In the former two cases, we cannot identify a concrete locally universal object, whereas the previous paragraph shows that $\cS$ is a locally universal operator system.

Since $\cR$ embeds into any II$_1$ factor, we see that CEP is equivalent to $\Th_\forall(M)=\Th_\forall(\cR)$ for each II$_1$ factor $\cR$.  This motivates us to ask the following:

\begin{question}
Is $\Th_\forall(\cS)$ the unique universal theory of infinite-dimensional operator systems?
\end{question}

\section{Tubularity}

In what follows, we will be considering maps $\rho:A\to B^\omega$.  Given such a map $\rho$, we define maps $\rho_i:A\to B$ in such a way that $\rho(a)=(\rho_i(a))^\bullet$.  Of course the maps $\rho_i$ are not uniquely defined.  If $\rho(1)=1$, we can (and will) always suppose that $\rho_i(1)=1$.

Following Jung \cite{Jung}, we make the following definition.

\begin{df} Let $A$ and $B$ be unital C$^*$-algebras. A u.c.p.\ map $\rho = (\rho_i)^\bullet : A\to B^\omega$ is \emph{tubular} if for any finite subset $F\subset A$ and any $\delta>0$ there exists $k$ so that there are u.c.p.\ maps $\phi_i: B\to M_k(\bC)$, $\psi_i: M_k(\bC)\to B$ so that $\|\rho_i(x) - (\psi_i\circ\phi_i)\circ\rho_i(x)\|<\delta$ for all $x\in F$ and for almost all $i$.
\end{df}

\begin{rmks}
\

\begin{enumerate}
\item Tubularity does not depend on the representative sequence $(\rho_i)^\bullet$.
\item Tubular maps are nuclear.
\item If $B$ is nuclear, then for any embedding $\theta: A\to B$, the ``diagonal'' embedding $\theta^\omega: A\to B^\omega$ obtained by composing $\theta$ with the diagonal embedding $B\to B^\omega$ is tubular.
\end{enumerate}
\end{rmks}

\begin{lemma}
Suppose that $A$ is separable and there exists $B$ and an elementary tubular embedding $\rho:A\to B^\omega$.  Then $A$ is nuclear.
\end{lemma}

\begin{proof}
We already showed above that if a C$^*$ algebra admits an existential nuclear embedding into another algebra, then it is nuclear.
\end{proof}

Nuclear maps on separable C$^*$-algebras are always liftable by a result of Choi and Effros (see \cite[6.1.4]{Ror}). We are able to offer an alternative quantitative proof in the tubular case.

\begin{prop}\label{lifting} Let $A$ be a unital, separable C$^*$-algebra. If the u.c.p.\ map $\rho=(\rho_i)^\bullet: A\to B^\omega$ is tubular, then there is a u.c.p.\ lift $(\tilde\rho_i: A\to B)$ which may furthermore be chosen so that $\tilde{\rho_i}$ is finite rank almost everywhere.
\end{prop}

\begin{proof} For $E$ a finite-dimensional unital operator subspace of $A$, let $\rho_i^E$ denote the restriction of $\rho_i$ to $E$.  It suffices to show that, for any such $E$ and any $\e>0$, there exists $r\geq 1$ and, for almost all $i$, u.c.p. maps $\sigma_i:E\to M_r(\bC)$ and $\tau_i:M_r(\bC)\to B$ so that $\|\tau_i\circ \sigma_i-\rho_i^E\|\leq \e$.  Indeed, we may then use Arveson extension to extend each $\sigma_i$ to a u.c.p. map $\tilde{\sigma}_i:A\to M_r(\bC)$ and then set $\tilde\rho_j^E = \tau_j\circ\tilde\sigma_j: A\to B$. Letting $E\subset A$ be an arbitrary finite-dimensional operator subspace and letting $\e$ tend to $0$, we have constructed the requisite u.c.p.\ lift $(\tilde\rho_i: A\to B)$ of $\rho$.

Fix once and for all $\e>0$ and $F\subset A$ a finite set of self-adjoint elements, and let $E\subset A$ be the finite-dimensional operator system spanned by $F$ and $1$.  Restricting $\rho_i$ to the finite-dimensional space $E$, we may assume without loss of generality that each $\rho_i^E$ is unital, linear, and $\ast$-linear. (Choose a basis $1, v_1, \dotsc, v_n$ of $E$, and replace $\rho_i$ with $\rho_i'$ which is the linear extension of $\rho_i'(1) = 1, \rho_i'(v_j)= (\rho_i(v_j) + \rho_i(v_j^*)^*)/2$. Since $\rho_i$ converges pointwise to a unital, $\ast$-linear map $\rho$, thus uniformly on bounded subsets of $E$, so does $\rho_i'$.)

Since $\rho$ itself is u.c.p.\ and $E$ is finite-dimensional, it must be the case that for each $n$, $\lim_\cU\|\rho_i^E\otimes\id_{M_n(\bC)}\|\leq \|\rho\|_{\cb} = 1$. Choose $I_n\in \omega$ so that $\sup_{i\in I_n} \|\rho_i^E\otimes\id_{M_n(\bC)}\|\leq 1+\e/2$. Since $\rho$ is tubular, there exists $r$ and $J\subset I_r$ generic so that there are u.c.p.\ maps $\phi_j: B\to M_r(\bC)$ and $\psi_j: M_r(\bC)\to B$ so that $\|(\psi_j\circ\phi_j)\circ\rho_j^E - \rho_j^E\|< \e/2$ for $j\in J$. Consider the map $\sigma_j':= \phi_j\circ\rho_j^E: E\to M_r(\bC)$. By Smith's lemma (\cite[Proposition 8.11]{Paulsen}) we have that \[\|\sigma_j'\|_{\cb} = \|\sigma_j'\otimes\id_{M_r(\bC)}\|\leq \|\phi_j\|_{\cb}\cdot\|\rho_j^E\otimes\id_{M_r(\bC)}\|\leq 1 + \e/2.\] By \cite[Lemma 6.1.7]{Ror}, for each $j\in J$, there exists u.c.p. $\sigma_j: E\to M_r(\bC)$ so that  $\|\sigma_j - \sigma_j'\|_{\cb}\leq \e/2$. Setting $\tau_j = \psi_j$, we see that $(\sigma_j, \tau_j)$ are the required pairs of u.c.p.\ maps.

\end{proof}

\begin{prop}\label{uniquetubular} Let $A$ be a unital, separable C$^*$-algebra, and let $B$ be unital, nuclear. Suppose $\rho: A\to B^\omega$ is the unique embedding up to unitary conjugacy. Then $\rho$ is tubular.
\end{prop}

\begin{proof} Let $1\in F\subset A$ be any finite subset of self-adjoint elements. We denote by $X_A(F, n,\e)$ the set of unital maps $\varphi: F\to B$ so that $\|p(\varphi(F))\| \sim_\e \|p(F)\|$ for any non-commutative polynomials $p(X) =\sum_i c_i X^{\alpha(i)}$ of degree $n$ in $|F|$ variables with $\sum_i |c_i|\leq 1$. Let $1\in F_1\subset F_2\subset\dotsb \subset F_n\subset\dotsb$ be a sequence of finite subsets of self-adjoint elements of $(A_{sa})_1$ so that $\bigcup_n F_n$ is dense in $(A_{sa})_1$ and generates $A$ as a C$^*$-algebra.

\

\noindent \textbf{Claim:} $\rho = (\rho_i)^\bullet$ is the unique embedding up to unitary conjugacy if and only if for any $\delta, m$, there exists $k, n, \e$ so that for any $l\geq k$ and $\xi, \eta\in X_A(F_l,n,\e)$ there is a unitary $u\in U(B)$ so that $\|u\xi(x)u^* - \eta(x)\|<\delta$ for all $x\in F_m$.

\

\noindent \textbf{Proof of Claim:} Note that two embeddings $(\rho_i)_\omega$ and $(\rho_i')_\omega$ of $A$ are unitarily conjugate in $B^\omega$ if and only if they are approximately unitarily conjugate if and only if for every $F\subset A$ finite and $\delta>0$ there is a sequence of unitaries $u_i\in U(B)$ so that $\lim_\omega \max_{x\in F} \|\rho_i(x) - u_i^* \rho_i'(x) u_i\| <\delta$. Consider two embeddings $\rho = (\rho_i)_\omega, \rho' = (\rho_i')_\omega$. Let $F$ be any finite subset of self-adjoint elements in $(A_{sa})_1$ and choose $m$ sufficiently large so that $F\subset_{\delta/2} F_m$. For the $k,n,\e$ corresponding to $m,\delta/2$ we have that $\rho_i,\rho_i'\in X_A(F_k,n,\e)$ for almost every $i$, so there exist unitaries $u_i\in U(B)$ so that $\|\rho_i(x) - u_i^*\rho_i'(x) u_i\|<\delta/2$ for all $x\in F_m$ and almost every $i$.  It follows that $\|\rho_i(x) - u_i^*\rho_i'(x) u_i\|<\delta$ for all $x\in F$ and almost every $i$.

Conversely, suppose there exist $m,\delta$ so that for all $k,n,\e$ there is $l\geq k$ and $\xi_l,\eta_l\in X_A(F_l,n,\e)$ so that $\inf_{u\in U(B)} \max_{x\in F_m} \|\xi_l(x) - u^*\eta_l(x) u\|\geq \delta$. Since $F_m$ is finite, by sparsfying, we can find a sequence $(l_p\geq p)$ so that there are $\xi_{l_p},\eta_{l_p} \in X_A(F_{l_p},p,1/p)$ so that $\inf_{u\in U(B)}\max_{F_m}\|\xi_{l_p}(x) - u^*\eta_{l_p}(x) u\|\geq \delta$. The embeddings $\xi = (\xi_{l_p})_{p\in\omega}$ and $(\eta_{l_p})_{p\in\omega}$ then cannot be unitarily conjugate.

\

Let $F$ be a finite subset of $(A_{sa})_1$. For such $\delta/8$, fix $\xi\in X_A(F_l,n,\e)$ for $l$ sufficiently large so that $F\subset_{\delta/8} F_l$. Then, since $\rho_i\in X_A(F_l,n,\e)$ for almost every $i$, there are unitaries $u_i\in B$ so that $\|u_i\xi(x)u_i^* - \rho_i(x)\|<\delta/8$. By the nuclearity of $B$, we can find for some $k$ u.c.p.\ maps  $\phi: B\to M_k(\bC)$ and $\psi: M_k(\bC)\to B$ so that $\|\xi(x) - (\psi\circ\phi)\circ\xi(x)\|<\delta/8$ for all $x\in F_l$. Setting $\phi_i(x) = \phi(u_i^* x u_i)$ and $\psi_i(x) = u_i\psi(x)u_i^*$, we obtain pairs $(\phi_i, \psi_i)$ which are an $(F,\delta)$-witness for tubularity. Since $F$ and $\delta$ were arbitrary, we are done.
\end{proof}

\begin{cor} Suppose that $A$ is a unital, separable C$^*$-algebra, $B$ is a nuclear C$^*$ algebra, and suppose that there is a unique embedding up to unitary conjugacy $A\to B^\omega$.  Then $A$ is exact.
\end{cor}

\begin{cor} Let $D$ be unital, separable, simple, and $\O_2$-stable. If $D$ has a unique embedding into $\O_2^\omega$ up to unitary conjugacy, then $D$ is elementarily equivalent to $\O_2$ if and only if $D\cong\O_2$.
\end{cor}

\begin{proof}

If $D$ is a separable model of $\Th(\O_2)$ and has a unique embedding into $\O_2^\omega$ up to unitary conjugacy, then that embedding is tubular and elementary (since $\O_2^\omega$ is an $\aleph_1$-saturated model of its theory), whence $D$ is nuclear and thus isomorphic to $\O_2$.
\end{proof}
%

\begin{question} Let $B$ be unital, separable, exact, (simple?). Does $B\otimes\O_2$ have a unique embedding into $\O_2^\omega$ up to unitary conjugacy? Are any two embeddings $B\to \O_2'\cap \O_2^\omega$ unitarily conjugate?
\end{question}



\section{Some model theory of $\O_2$}

In this section, we use our earlier results to initiate a model theoretic study of the theory of $\O_2$.  First, recall that a structure $M$ is the prime model of its theory if it embeds elementarily into all models of $\Th(M)$.

\begin{prop}
$\mathcal O_2$ is the prime model of its theory.
\end{prop}

\begin{proof}
Suppose that $A\models \Th(\mathcal O_2)$ is separable.
Then $A$ is $\O_2$-stable, whence $A$ contains a copy of $\mathcal O_2$.  We then have $\mathcal O_2\subseteq A \hookrightarrow \mathcal O_2^\omega$, where the embedding of $A$ into $\mathcal O_2^\omega$ is elementary; since the composition $\mathcal O_2\hookrightarrow \mathcal O_2^\omega$ is elementary, we get that the embedding $\mathcal O_2\subseteq A$ is elementary.
\end{proof}

Notice that $\mathcal O_2$ is not a \emph{minimal} model of its theory, that is, it has proper elementary submodels.  Indeed, the embedding $\id\otimes 1\mathcal: \O_2\hookrightarrow \mathcal O_2\otimes \mathcal O_2$ is elementary.

\begin{cor}
Every elementary submodel of $\mathcal O_2$ is isomorphic to $\mathcal O_2$.
\end{cor}

\begin{proof}
An elementary submodel of $\mathcal O_2$ is necessarily a prime model; since prime models are unique up to isomorphism, we are done.
\end{proof}

We can say even more:

\begin{prop}
Suppose that $A$ is a subalgebra of $\mathcal O_2$ that is e.c.\   in $\mathcal O_2$.  Then $A$ is isomorphic to $\mathcal O_2$ and $A\preceq \mathcal O_2$.
\end{prop}

\begin{proof}
Since being $\mathcal O_2$ stable is $\forall\exists$-axiomatizable (Fact \ref{O2stableaxiom}) and $\mathcal O_2$ is $\mathcal O_2$-stable, we have that $A$ is $\mathcal O_2$-stable.  On the other hand, $A$ is exact and e.c.\   in $\mathcal O_2$, whence $A$ is simple (Fact \ref{O2stableaxiom}) and nuclear (Theorem \ref{exactec2}) and thus $A\otimes \mathcal O_2\cong \mathcal O_2$.  It follows that $A\cong \mathcal O_2$.  Since any embedding of $\mathcal O_2$ into itself is elementary, we have that $A\preceq \mathcal O_2$.
\end{proof}

Recall that a theory is $\omega$-categorical if it has a unique (up to isomorphism) separable model.

\begin{prop}
$\Th(\O_2)$ is not $\omega$-categorical.  In fact, there must exist a nonexact separable model of $\Th(\O_2)$.
\end{prop}

\begin{proof}
Let $A:=C^*(\mathbb F_2)$.  Since $A$ is quasidiagonal, $A$ is $\O_2^\omega$-embeddable.  Without loss of generality, assume that $A$ is a subalgebra of $\O_2^\omega$.  Let $\O_2'$ be a separable elementary substructure of $\O_2^\omega$ containing $A$.  Then $\O_2'$ is not exact, else $A$ is exact.
\end{proof}

\begin{rmk}
If KEP holds, then there are uncountably many nonisomorphic separable models of $\Th(\O_2)$.
Indeed, suppose that there are only countably many separable models of $\Th(\O_2)$ up to isomorphism.  KEP implies that every separable C$^*$ algebra embeds into some such model, whence as in the proof of Proposition \ref{uncountable} their tensor product is a universal separable e.c.\ C$^*$ algebra.
\end{rmk}

\begin{question}\label{exactmodel}
Can $\Th(\O_2)$ have an exact, nonnuclear model?  (We already know that $\O_2$ is the only separable, nuclear model of its theory.)  More generally, do the axioms for being simple and $\O_2$-stable axiomatize $\Th(\O_2)$?
\end{question}

We suspect the answer to both of these questions is:  no. In connection with this, we now establish that a positive answer to the latter question implies a strong form of KEP is true.

\begin{lemma}\label{ee-nuclear}
Suppose that $A$ is a simple, $\mathcal O_2$-stable C$^*$ algebra that is elementarily equivalent to a nuclear algebra.  Then $A\equiv \O_2$.
\end{lemma}

\begin{proof}
Suppose that $A\equiv N$ where $N$ is nuclear.  Since $A$ is simple and purely infinite, we have that $N$ is simple and purely infinite.  Passing to a separable elementary substructure of $N$ (which preserves the nuclearity of $N$ by Theorem \ref{exactec2}), we may further assume that $N$ is separable.  Since $N$ is separable, simple, and nuclear, we have that $N\otimes \O_2\cong \O_2$; since $N$ is $\O_2$-stable, we conclude that $N\cong \O_2$.
\end{proof}

\begin{prop}\label{ee-kep}
KEP is equivalent to the statement that there is an e.c.\ C$^*$ algebra that is elementarily equivalent to a nuclear algebra.
\end{prop}

\begin{proof}
The forward direction follows from our earlier characterization of KEP, namely that $\O_2$ is e.c.  Now suppose that $A$ is e.c., $N$ is nuclear, and $A\equiv N$.  By the Downward L\"owenheim--Skolem theorem, we can assume that $A$ is separable.  Since $A$ is simple and $\O_2$-stable, it follows from Lemma \ref{ee-nuclear} that $A\equiv \O_2$.  Since $\O_2$ is a prime model of its theory, $\O_2$ embeds elementarily into $A$; since $A$ is e.c., it follows that $\O_2$ is e.c., whence KEP holds.
\end{proof}

\begin{cor}\label{strong-kep} The following statements are successively weaker:

\begin{enumerate}
\item simplicity and $\O_2$-stability axiomatize $\Th(\O_2)$;
\item every e.c.\ C$^*$-algebra has the same theory as $\O_2$;
\item KEP is true.
\end{enumerate}

\end{cor}

\begin{proof} The former implication is trivial, while the latter implication follows directly from the previous proposition.
\end{proof}


We remark that Lemma \ref{ee-nuclear} has some interest independent of its use in the proof of Proposition \ref{ee-kep}. In \cite{AF}, the authors ask if every C$^*$ algebra is elementarily equivalent to a nuclear one.  In \cite{FHRTW}, they give examples of C$^*$ algebras which are not elementarily equivalent to nuclear algebras.  Moreover, all of their counterexamples admit traces.  In \cite{Fnotes}, the author asks for an example of an exact, nonnuclear C$^*$ algebra that is not elementarily equivalent to a nuclear one.
Thus, if $A$ is a simple, $\O_2$-stable C$^*$ algebra that is not elementarily equivalent to $\O_2$, then $A$ is not elementarily equivalent to a nuclear algebra, providing a completely different type of counterexample to the question posed in \cite{AF}.  Moreover, if $A$ is also exact, then we also have a solution to the question posed in \cite{Fnotes}.

\

Recall that a theory is model-complete if all of its models are existentially closed models of $T$.  In \cite{nomodcomp}, the authors show that if CEP holds, then $\Th(\R)$ is not model-complete; in \cite{ecfactor}, the authors show that $\Th(\R)$ is not model-complete without assuming CEP.  We now show that KEP implies $\Th(\O_2)$ is not model-complete.  The argument follows the same strategy as the argument in \cite{ecfactor}; the dependence on KEP is needed to ensure that taking a crossed product by $\z$ retains $\O_2^\omega$-embeddability.

\begin{prop}
If KEP holds, then $\Th(\O_2)$ is not model-complete.
\end{prop}

\begin{proof}
Let $B$ be a separable model of $\Th(\O_2)$.  We claim that all embeddings of $B$ into $\O_2^\omega$ are unitarily conjugate.  If this is the case, then by Proposition \ref{uniquetubular}, the unique embedding (up to unitary conjugacy) is tubular.  Since this embedding is also elementary, we get that $B$ is nuclear, whence isomorphic to $\O_2$, contradicting that $\Th(\O_2)$ is not $\omega$-categorical.

We now prove the claim.  Suppose that $f,g:B\to \O_2^\omega$ are embeddings.  Let $(b_i)_{i\in \n}$ be a countable dense subset of $B$.  It suffices to find, for any $i\in \n$, a unitary $u$ in $\O_2^\omega$ such that $d(u^*f(b_j)u,g(b_j))<\frac{1}{i}$ for all $j\leq i$. Since $f$ and $g$ are elementary (as $\Th(\O_2)$ is model-complete), we have that $\tp^{\O_2^\omega}(f(b_1)\cdots f(b_i))=\tp^{\O_2^\omega}(g(b_1)\cdots g(b_j))$, whence there is an elementary extension $\O_2'$ of $\O_2^\omega$ and an automorphism $\alpha$ of $\O_2'$ such that $\alpha(f(b_j))=g(b_j)$ for $j\leq i$.  By model-completeness, $\O_2'$ is an e.c.\   model of its theory; by KEP, we have that $\alpha$ is approximately inner, and by elementarity, the desired unitary exists in $\O_2^\omega$.
\end{proof}

Recall that a model complete theory $T$ is the \emph{model companion} of another theory $T_0$ if every model of $T$ embeds in a model of $T_0$ and every model of $T_0$ embeds in a model of $T$.  If $T_0$ is universally axiomatizable (e.g., when $T_0$ is the theory of unital C$^*$ algebras), then the model companion of $T_0$ exists if and only if the class of e.c.\   models of $T$ is an axiomatizable class.

\begin{cor}
KEP implies that the theory of unital C$^*$ algebras does not have a model companion.
\end{cor}

\begin{proof}
Suppose that $T$ is a model companion for the theory of C$^*$ algebras.  Since the class of unital C$^*$ algebras has the amalgamation property, $T$ has quantifier-elimination.  Since any two models of $T$ have a common substructure, namely $\mathbb C$, we see that $T$ is complete.  By KEP, $\O_2$ is e.c., whence a model of $T$.  Thus, $T=\Th(\O_2)$, contradicting the fact that KEP implies that $\Th(\O_2)$ is not model-complete.
\end{proof}


There are two goals that we have not achieved outright:  proving that $\Th(\O_2)$ is not model-complete and showing that $\O_2$ is the only exact model of its theory.  The following proposition shows that at least one of these goals is going to be achieved.

\begin{prop}\label{goodnews}
If $\Th(\O_2)$ is model-complete, then $B\otimes \O_2\not\equiv \O_2$ for any simple, exact, nonnuclear $B$.
\end{prop}

\begin{proof}
Suppose that $B$ is simple, exact, and nonnuclear.  We then have that $B\otimes \O_2$ embeds into $\O_2$.  If $B\otimes \O_2$ were elementarily equivalent to $\O_2$, then that embedding would be elementary, whence, as $B\otimes \O_2$ is simple, we would have that $B\otimes \O_2$ is nuclear as well, whence $B$ would be nuclear, a contradiction.  (To see that $B$ is nuclear, just compose with a slice map.  Alternatively, if $B\otimes \O_2$ is nuclear and elementarily equivalent to $\O_2$, then $B\otimes \O_2\cong \O_2$, whence $B$ is nuclear.)
\end{proof}

\begin{prop}
Suppose that $A$ is a simple, purely infinite, $\O_2^\omega$-embeddable C$^*$ algebra with trivial $K_0$ such that $\Th(A)$ is $\forall\exists$-axiomatizable (e.g., $\Th(A)$ is model complete).  Then $A\equiv \O_2$.
\end{prop}

\begin{proof}
This follows from the usual ``sandwiching'' trick.  By \cite[Proposition 4.2.3(ii)]{Ror}, we have that $\O_2\hookrightarrow A$.  Now embed $A\hookrightarrow \O_2^\omega$ and do the union of chains argument as in, for example, Proposition 3.2 in \cite{nomodcomp}.  The union will be elementarily equivalent to both $\O_2$ and $A$.
\end{proof}

We now offer one possible approach for obtaining an answer to Question \ref{exactmodel}.  Let us say that a C$^*$ algebra $A$ is \emph{stably presented} if there are a finite set of generators $\mathcal G$ and a finite set of weakly stable relations $\R$ such that $A=C^*(\mathcal G | \R)$.  (See \cite[II.8.3]{Black})

\begin{lemma}
Suppose that $A$ is simple, stably presented, and $B^\omega$-embeddable for some nuclear $B$.  Then $A$ is exact.
\end{lemma}

\begin{proof}
By assumption, there is a $\ast$-homomorphism $A\to B$.  Since $A$ is simple, this $\ast$-homomorphism is an embedding; since $B$ is nuclear, it follows that $A$ is exact.
\end{proof}

\begin{prop}
$\O_2$ is the only stably presented model of its theory.
\end{prop}

\begin{proof}
Suppose that $A$ is a stably presented model of $\Th(\O_2)$.  It is enough to show that any two embeddings $f,g:A\to \O_2^\omega$ are unitarily conjugate, for then the unique (elementary) embedding of $A$ into $\O_2^\omega$ is tubular, whence $A$ is nuclear and hence isomorphic to $\O_2$.  Let $x_1,\ldots,x_k$ denote the generators of $A$ and, for $i=1,\ldots,k$, let $(y_{i,n})^\bullet,(z_{i,n})^\bullet\in \O_2^\omega$ denote $f(x_i)$ and $g(x_i)$.  By weak stability, we may assume that, for each $n$, $y_{1,n},\ldots,y_{k,n}$ generates a copy of $A$ in $\O_2$; likewise, that $z_{1,n},\ldots,z_{k,n}$ generates a copy of $A$ in $\O_2$.  Since $A$ is simple, we get embeddings $f_n,g_n:A\to \O_2$ obtained by defining $f_n(x_i):=y_{i,n}$, $g_n(x_i):=z_{i,n}$.  Moreover, by correcting that each $f_n$ and $g_n$ send $1_A$ to a projection that is within $1$ of the identity of $\O_2$, we may assume that each $f_n$ and $g_n$ are unital embeddings.  Since any two unital embeddings of $A$ in $\O_2$ are approximately unitarily conjugate (see \cite[Theorem 6.3.8(ii)]{Ror}), we get unitaries $u_n\in \O_2$ such that, for each $i=1,\ldots,k$, we have $(u_n^*)^\bullet (y_{i,n})^\bullet (u_n)^\bullet =(z_{i,n})^\bullet$.  It follows that $f$ and $g$ are unitarily conjugate.
\end{proof}

\begin{question}
If $A$ is stably presented, is $A\otimes \O_2$ stably presented?
\end{question}

\begin{question}
Is there a simple, stably presented, nonnuclear $A$?
\end{question}

If both of these questions have affirmative answers, then for $A$ simple, stably presented, nonnuclear, we know that $A\otimes \O_2$ is simple, $\O_2$-stable, and stably presented, whence not a model of $\Th(\O_2)$, answering Question \ref{exactmodel}.

As a caveat to the above questions:  note that if $A$ is simple, stably presented, and $\cS^\omega$-embeddable (e.g., quasi-diagonal), then $A$ is a matrix algebra.

\appendix
\section{Model-theoretic forcing}

We work in a countable signature $\la$ and add countably many  new constant symbols $C$ to the language.  An $\la(C)$-structure $M$ is said to be \emph{canonical} if $\{c^M \ : \ c\in C\}$ is dense in $M$.

  We fix a class $\k$ of structures and let $\k(C)$ denote the class of all structures $(M,a_c)_{c\in C_0}$, where $C_0$ is a finite subset of $C$.  We treat such structures as $\la(C_0)$-structures in the natural way.

Conditions are now finite sets of the form $\{\varphi_1<r_1,\ldots,\varphi_n<r_n\}$, where each $\varphi_i$ is an atomic $\la(C)$-sentence and there is $M\in \k(C)$ such that $\varphi_i^M<r_i$ for each $i=1,\ldots,n$.  The partial order on conditions is reverse inclusion.  If $p$ is a condition and $\varphi$ is an atomic sentence of $\la(C)$, we define $f_p(\varphi):=\min\{r\leq 1 \ | \varphi<r\in p\}$, with the understanding that $\min(\emptyset)=1$.  For a condition $p$ and an $\la(C)$-sentence $\varphi$, we define the value $F_p(\varphi)\in [0,1]$ by induction on $\varphi$.
\begin{itemize}
\item $F_p(\varphi)=f_p(\varphi)$ if $\varphi$ is atomic.
\item $F_p(\neg \varphi)=\neg \inf_{q\supseteq p}F_q(\varphi)$.
\item $F_p(\frac{1}{2}\varphi)=\frac{1}{2}(\varphi)$.
\item $F_p(\varphi+\psi)=F_p(\varphi)+F_p(\psi)$.  (Truncated addition)
\item $F_p(\inf_x\varphi(x))=\inf_{c\in C}F_p(\varphi(c))$.
\end{itemize}

If $r\in \r$ and $F_p(\varphi)<r$, we say that $p$ \emph{forces} that $\varphi<r$, and write $p\Vdash \varphi<r$.

\begin{df}
We say that a nonempty set $G$ of conditions is \emph{generic} if the union of two elements of $G$ is once again an element of $G$ and for every $\la(C)$-sentence $\varphi$ and every $r>1$, there is $p\in G$ such that $F_p(\varphi)+F_p(\neg \varphi)<r$.
\end{df}

It is proven in \cite{BI} that generic sets always exist.  If $G$ is generic and $\varphi$ is an $\la(C)$-sentence, set $\varphi^G:=\inf_{p\in G}F_p(\varphi)$.
\begin{thm}[Generic Model Theorem]
Let $M_0^G$ denote the term algebra $\mathcal T(C)$ equipped with the natural interpretation of the function symbols and interpreting the predicate symbols by $P^{M_0^G}(\vec \tau):=P(\vec \tau)^G$.  Let $M^G$ be the completion of $M_0^G$.  Then $M^G$ is an $\la(C)$-structure such that, for all $\la(C)$-sentences $\varphi$, we have $\varphi^{M^G}=\varphi^G$.
\end{thm}

We say that an $\la$-structure $N$ is \emph{finitely generic for $\k$} if there is a generic set $G$ of conditions such that $M$ is isomorphic to the $\la$-reduct of $M^G$.  Note that finitely generic structures exist as generic sets of conditions exist.

The following is a special case of the Omitting Types Theorem proven in \cite{BI}.  By a \emph{finite piece of $\k$} we mean a finite set of inequalities of the form $\varphi<r$ where $\varphi$ is quantifier-free and such that the set of inequalities is satisfiable in some structure in $\k$.
\begin{thm}[Omitting Types Theorem]
Let $(\varphi_n \ | \ n<\omega)$ be a sequence of $\la$-sentences such that, for each $n<\omega$, $\varphi_n$ is of the form
$$\sup_{x_1}\cdots\sup_{x_{m(n)}}\psi_n(x_1,\ldots,x_{m(n)}),$$ where $$\psi_n(x_1,\ldots,x_{m(n)}):=\bigwedge_{k<\omega}\inf_{y_1}\cdots \inf_{y_i(n,k)}\sigma_{n,k}(\vec x_n,\vec y_{n,k}),$$ and where $\sigma_{n,k}$ is quantifier-free, $\vec x_n=x_1,\ldots,x_{m(n)}$, and $\vec y_{n,k}=y_1,\ldots,y_{i(n,k)}$.  Let $(r_n \ | \ n<\omega)$ be a sequence of real numbers such that, for every finite piece $p$ of $\k$ and every $\vec {c}_n\in C^{m(n)}$, the set $p\cup \{\psi_n(\vec c_n)<r_n\}$ is satisfiable in $\k$.  Then there exists a canonical $\la(C)$-structure $M$ such that $\varphi_n^M\leq r_n$ for every $n<\omega$.  In fact, $M$ is of the form $M^G$ for some generic set $G$.
\end{thm}

In Section 3, we needed to know that generic models of $\forall\exists$-axiomatizable theories are e.c.\   models of the theory.  The remainder of this appendix is devoted to establishing this fact.

It will convenient for this discussion to slightly alter the definition of atomic diagram of $M$.  We set $D(M)$ to be the set of conditions $\sigma<r$, where $\sigma$ is an atomic $L(M)$-sentence with $\sigma^M=0$ and $r\in (0,1]$.

\begin{lemma}\label{aefingen}
If $T$ is a $\forall\exists$-axiomatizable theory and $\k$ is the class of models of $T$, then a structure finitely generic for $\k$ belongs to $\k$.
\end{lemma}

\begin{proof}
Suppose that $\sup_x\inf_y\varphi(x,y)=0$ is an axiom of $T$.  Suppose that $M$ is isomorphic to the $\la$-reduct of $M^G$.  Fix $c\in C$; it is enough to show that $(\inf_y\varphi(c,y))^{M^G}=0$.  Suppose, towards a contradiction, that $(\inf_y\varphi(c,y))^{M^G}=\epsilon>0$.  Take $p\in G$ such that $$F_p(\inf_y\varphi(c,y))+F_p(\neg \inf_y\varphi(c,y))<1+\frac{\epsilon}{2}.$$  It follows that $\inf_{q\supseteq p}F_q(\inf_y\varphi(c,y))\geq \frac{\epsilon}{2}$.  Now take $N\models T\cup p$.  Fix $d\in C$ not occurring in $p$.  By redefining the interpretation of $d$ in $N$ if necessary, we may assume that $\varphi(c,d)^N<\frac{\epsilon}{2}$ (since $N\models T$).  There is then a finite fragment $q$ of $D(N)$ such that $q\models \varphi(c,d)<\frac{\epsilon}{2}$; since $N\models p$, we may as well assume that $p\subseteq q$.  Suppose that $G'$ is generic set containing $q$; then $\varphi(c,d)^{M^{G'}}<\frac{\epsilon}{2}$.  It follows that there is $q'\supseteq q$ from $G'$ such that $F_{q'}(\varphi(c,d))<\frac{\epsilon}{2}$, a contradiction.
\end{proof}

Actually the above proof showed the following fact:

\begin{cor}
For any theory $T$, if $M$ is finitely generic for the class of models of $T$, then $M\models T_{\forall\exists}$.
\end{cor}
It will be convenient to introduce the analog of weak forcing:

\begin{df}
$F_p^w(\varphi)=\sup_{q\supseteq p}\inf_{q'\supseteq q}F_{q'}(\varphi)$.
\end{df}
\begin{lemma}
If $T$ is $\forall\exists$-axiomatizable and $M$ is finitely generic for the class of models of $T$, then $M$ is an e.c.\   model of $T$.
\end{lemma}

\begin{proof}
Suppose that $M\subseteq N$ with $N\models T$ and $\varphi(x)$ is an $L(C)$-formula.  Suppose that $(\inf_x \varphi(x))^N=r$.  Without loss of generality, $r<1$.  Fix $\epsilon>0$.  We must show that $(\inf_x \varphi(x))^M< r+\epsilon$.  Without loss of generality, we may assume that $M$ is the reduct of $M^G$ for some generic $G$.  Suppose, towards a contradiction, that $(\inf_x\varphi(x))^M\geq r+\epsilon$.  Fix $p\in G$ such that $$F_p(\inf_x\varphi(x))+F_p(\neg \inf_x\varphi(x))<1+\frac{\epsilon}{2}.$$  Our assumption then implies that $F_p(\neg \inf_x\varphi(x))<1-r-\frac{\epsilon}{2}$, whence $\inf_{q\supseteq p}F_q(\inf_x\varphi(x))\geq r+\frac{\epsilon}{2}$.

Pick a constant $c$ not appearing in either $p$ or $\varphi$.  Consider the $L(C)$-structure $N'$ obtained from $N$ by interpreting constants appearing in $p$ and $\varphi$ as $M$ does while interpreting $c$ in such a way so that $\varphi(c)^{N'}<r+\frac{\epsilon}{2}$.  Notice then that $D(N')\cup T\models \varphi(c)<r+\frac{\epsilon}{2}$.  By our choices and compactness, we may choose a finite $q\subseteq D(N')$ with $p\subseteq q'$ such that $q\cup T\models \varphi(c)<r+\frac{\epsilon}{2}$.  We claim that $F_q^w(\inf_x \varphi(x))<r+\frac{\epsilon}{2}$.  Indeed, fix $q'\supseteq q$.  Consider generic $G'$ containing $q'$.  By Lemma \ref{aefingen}. $M^{G'}\models T$, whence we have $\varphi(c)^{M^{G'}}<r+\frac{\epsilon}{2}$.  This  implies that there is $q''\in G$ containing $q'$ such that $F_{q''}(\varphi(c))<r+\frac{\epsilon}{2}$.

We have now arrived at the desired contradiction as $$F_q^w(\inf_x\varphi(x))\geq \inf_{q'\supseteq q}F_{q'}(\inf_x\varphi(x))\geq \inf_{q'\supseteq p}F_{q'}(\inf_x\varphi(x))\geq r+\frac{\epsilon}{2}.$$
\end{proof}

\section{Operator spaces and systems in continuous logic}

We first show how to axiomatize operator systems in continuous logic.  We work in the version of continuous logic presented in \cite{MTOA2}.  We think of an operator system as a many sorted structure with sorts:
\begin{itemize}
\item $M_n(\S)$, $n\geq 1$, for the matrix algebras over $\S$ with domains of quantification intended to be with respect to the matrix norm induced by the matrix order;
\item $\C_n$, $n\geq 1$, for the cones of positive elements, with domains of quantification those inherited from $M_n(\S)$;
\item $M_{m,n}(\mathbb C)$ for matrices over the complex numbers with domains of quantification with respect to the operator norm;
\item $\mathbb \r^{\geq 0}$ with domains of quantification as usual.
\end{itemize}

Note that we use a separate sort for the algebras over $\S$ as the metric on them is not merely the max metric induced by the metric on $\S$.

Here are the symbols:
\begin{itemize}
\item The constant symbols $0,1\in \S$ with domain $D_1(\S)$;
\item The function symbols $$+:M_n(\S)\times M_n(\S)\to M_n(\S) \text{ and }*:M_n(\S)\to M_n(\S).$$
\item The function symbols $l:\mathbb C\times M_n(\S)\to M_n(\S)$.
\item The function symbols $f_{m,n}:M_{m,n}(\mathbb C)\times M_m(\S)\to M_n(\S)$ (intended to be for $f_{m,n}(A,X):=A^*XA$).
\item The function symbols $+:\C_n\to \C_n$, $k:\r^{\geq 0}\times \C_n\to \C_n$ and $$g_{m,n}:M_{m,n}(\mathbb C)\times \C_m\to \C_n.$$
\item Inclusions $i_n:\C_n\to M_n(\S)$.
\item Function symbols $h_n:M_n(\S)\to M_{2n}(\S)$, intended to be for $$h_n(X):=\left(\begin{matrix} I & X \\ X^* & I \end{matrix}\right).$$
\item Function symbols $\pi_n^{ij}:M_n(\S)\to \S$, intended to be projection functions.  Likewise for scalar matrices.
\item Predicate symbols $\|\cdot \|_n:M_n(\S)\to \r^{\geq 0}$.
\item Function symbols for scalar matrix addition, multiplication and adjoint.  Constant symbols for $1$ and $i$.
\end{itemize}

In order to axiomatize operator systems in this language, we will need to know the following two easy facts:

\begin{lemma}\label{easy}
Suppose that $\S$ is a concrete operator system.
\begin{enumerate}
\item If $A\in M_n(\S)$, then
$$\|A\|\dotminus 1\leq d\left( \left(\begin{matrix} I & A \\ A^* & I \end{matrix}\right),\C_{2n}\right)\leq \sqrt{2}(\|A\|\dotminus 1).$$
\item If $A\in \C_n$, then $d(A,-\C_n)=\|A\|$.
\end{enumerate}
\end{lemma}

\begin{proof}
(1)  This basically follows from a slightly more careful analysis of the proof of \cite[Lemma 3.1(i)]{Paulsen}.  Indeed, for $x,y\in H$, we have
$$\left\langle \left(\begin{matrix} I & A \\ A^* & I\end{matrix}\right)\left(\begin{matrix} x \\ y \end{matrix}\right),\left(\begin{matrix}x \\ y\end{matrix}\right)\right\rangle =\langle x,x\rangle+\langle Ay,x\rangle+\langle x,Ay\rangle +\langle y,y\rangle.$$  Now if $\|A\|\leq 1$, then the right hand side of the above display is positive.  We may thus suppose that $\|A\|>1$.  Fix $\epsilon>0$ such that $\|A\|-\epsilon>1$ and take $y\in H$ with $\|y\|=1$ such that $\|Ay\|\geq \|A\|-\epsilon$.  Set $x:=-Ay/\|Ay\|$ so that $\langle Ay,x\rangle=\|Ay\|$.  Fix $E\in \C_{2n}$.  We then have $$\left\langle \left( \left(\begin{matrix} I & A \\ A^* & I\end{matrix}\right)-E\right)\left(\begin{matrix} x \\ y \end{matrix}\right),\left(\begin{matrix}x \\ y\end{matrix}\right)\right\rangle \leq 2-2\|Ay\|.$$  On the other hand, $$\left|\left\langle \left(\left(\begin{matrix} I & A \\ A^* & I\end{matrix}\right)-E\right)\left(\begin{matrix} x \\ y \end{matrix}\right),\left(\begin{matrix}x \\ y\end{matrix}\right)\right\rangle\right |\leq 2\left \|\left(\begin{matrix} I & A \\ A^* & I\end{matrix}\right)-E\right\|.$$  Thus, $$\left\|\left(\begin{matrix} I & A \\ A^* & I\end{matrix}\right)-E\right\|\geq \|Ay\|-1\geq \|A\|-\epsilon-1.$$  The first inequality now follows by letting $\epsilon$ go to $0$.  For the other inequality, set $r:=\|A\|$ and note that $\left( \begin{matrix} I &r^{-1}A \\ r^{-1}A^* & I \end{matrix}\right)\in \C_{2n}$ and that $$d\left(\left(\begin{matrix}I & A \\ A^* & I\end{matrix}\right),\left(\begin{matrix} I & r^{-1}A \\ r^{-1}A^* & I\end{matrix}\right)\right)=\left\|\left(\begin{matrix}0 & (1-r^{-1})A\\ (1-r^{-1})A^* & 0\end{matrix}\right)\right\|\leq \sqrt{2}(r-1).$$

For (2), suppose that $A,B\in \C_n$.  Then, since $A+B$ is hermitian, we have that
$$d(A,-B)=\|A+B\|=\sup\{\langle (A+B)x,x\rangle  \ : \ \|x\|=1\}.$$  Fix $\epsilon>0$ and take $x\in H$ such that $\langle Ax,x\rangle>\|A\|-\epsilon$.  We then have
$$d(A,-B)\geq \langle \langle (A+B)x,x\rangle\geq \|A\|-\epsilon+\langle Bx,x\rangle\geq \|A\|-\epsilon.$$  We thus have $d(A,-\C_n)\geq \|A\|$.  But $0\in \C_n$, so $d(A,-\C_n)=\|A\|$.
\end{proof}

Here are the axioms for $T_{\opsys}$:

\begin{itemize}
\item Each $M_n(\S)$ is a $*$-v.s. and the operations on $M_n(\S)$ are those induced from $\S$.  For example, $\pi^{ij}_n(A+B)=\pi^{ij}_n(A)+\pi^{ij}_n(B)$ and $\pi^{ij}_n(A^*)=\pi^{ji}_n(A)^*$.
\item Axioms saying that each $f_{m,n}$ is interpreted as it should be.
\item Axioms saying that each $i_n$ is an isometric inclusion.
\item $\sup_{A\in \C_n}d(i_n(A),i_n(A)^*)=0$.
\item Axioms saying that the operations of addition and scalar multiplication on $\C_n$ agree with those in $M_n(\S)$.
\item Axioms saying that each $h_n$ is interpreted as it should be.
\item $$\sup_{A\in M_n(\S)}\max\left((\|A\|_n\dotminus 1)\dotminus d(h_n(A),\C_{2n}),d(h_n(A),\C_{2n})\dotminus \sqrt{2}(\|A\|_n\dotminus 1)\right)=0.$$
\item $\sup_{A\in \C_n}\sup_{B\in \C_n}(\|A\|\dotminus d(A,-B))=0$.
\item Axioms saying that $\|\cdot\|_n$ is a seminorm.
\item $\sup_{A,B\in M_n(\S)}|d(A,B)-\|A-B\|_n|=0$.
\end{itemize}

Note that a model of the axioms is naturally a matrix ordered $*$-vector space.  Moreover, it follows from the axioms that $1$ is a matrix order unit.  Indeed, given $X\in M_n(\S)_h$, there is $r\in \R^{>0}$ such that $\|X\|_n\leq r$.  It follows that $\|\frac{1}{r}X\|_n\leq 1$, so $h_n(\frac{1}{r}X)\in \C_{2n}$ whence $\left(\begin{matrix} rI & X \\ X & rI\end{matrix}\right)\in \C_{2n}$.  By conjugating with the vector of all $I$'s, we have that $n(rI+X)\in \C_n$, whence it follows that $rI+X\in \C_n$.

It is not immediately clear that we get that $1$ is an \emph{archimedean} matrix order unit.  However, having an archimedean matrix order unit is only used to show that the seminorms induced from the matrix order are actually norms and that each $\C_n$ is closed in the topology induced by the norm.  However, the first fact follows from the connection between the seminorm and the metric and the latter fact follows from the fact that $\C_n$ is required to be complete as it is a sort in a metric structure.

Recall that an abstract operator system is a matrix-ordered $*$-vector space with an archimedean matrix order unit.  The Choi-Effros Theorem (see \cite[Theorem 13.1]{Paulsen}) says that the notions of abstract operator system and concrete operator system coincide.
\begin{thm}
Every (abstract) operator system $\S$ is, in a natural way, a model $M(\S)$ of $T_{\opsys}$.  Conversely, every model of $T_{\opsys}$ is an operator system.  The map $\S\mapsto M(S)$ is an equivalence of categories between the category of operator systems with complete order isomorphisms and the category of models of $T_{\opsys}$ with (model-theoretic) isomorphisms.
\end{thm}

\begin{proof}
By Lemma \ref{easy}, every concrete operator system is naturally a model of $T_{\opsys}$;  now use Choi-Effros.  Conversely, a model of $T_{\opsys}$ satisfies every axiom for being an abstract operator system except perhaps for $1$ being an archimedean matrix order unit.  However, since each $\C_n$ is closed, we can run the proof of Choi-Effros and get that the model is completely order isomorphic to a concrete operator system, whence $1$ must have been archimedean.  The rest is obvious.
\end{proof}

For operator spaces, one plays a similar, but even easier game.  We have sorts for

\begin{itemize}
\item $M_{m.n}(V)$, $m,n\geq 1$, for the matrix algebras over $V$ with domains of quantification intended to be with respect to the matrix norms;
\item $M_{p,q}(\mathbb C)$ for matrices over the complex numbers with domains of quantification with respect to the operator norm;
\end{itemize}

Here are the symbols:
\begin{itemize}
\item The constant symbol $0 \in V$ with domain $D_1(V)$;
\item The function symbol $+:M_{m,n}(V)\times M_{m,n}(V)\to M_{m,n}(V)$.
\item The function symbols $l:\mathbb C\times M_{m,n}(V)\to M_{m,n}(V)$.
\item The function symbols $$f_{p,m,n,q}:M_{p,m}(\mathbb C)\times M_{m,n}(V)\times M_{n,q}(\mathbb C)\to M_{p,q}(V)$$ (intended to be for $f_{p,m,n,q}(A,X,B):=AXB$).
\item Function symbols $\pi_{m,n}^{ij}:M_{m,n}(V)\to V$, intended to be projection functions.  Likewise for scalar matrices.
\item Function symbols $\oplus_{m,n,p,q}:M_{m,n}(V)\times M_{p,q}(V)\to M_{m+p,n+q}(V)$.
\item Predicate symbols $\|\cdot \|_{m,n}:M_{m,n}(V)\to \r^{\geq 0}$.
\item Predicate symbols $\|\cdot \|_{m,n}:M_{m,n}(\c)\to \r^{\geq 0}$.
\item Function symbols for scalar matrix addition, multiplication and adjoint.  Constant symbols for $1$ and $i$.
\end{itemize}

Here are the axioms $T_{\operatorname{mns}}$ for matrix normed spaces:

\begin{itemize}
\item Each $M_{m,n}(V)$ is a vector space and the operations on $M_{m,n}(V)$ are those induced from $V$.
\item Axioms saying that each $f_{p,m,n,q}$ is interpreted as it should be.
\item Axioms saying that $\|\cdot\|_{m,n}$ is a seminorm.
\item $\sup_{A,B\in M_{m,n}(V)}|d(A,B)-\|A-B\|_{m,n}|=0$.
\item Axioms that tell us the norm of an element of $\c^n$.
\item $\sup_{A\in M_{m,n}(\c)}\sup_{x\in \c^n}(\|Ax\|\dotminus (\|A\|\cdot \|x\|))=0$.
\item For each $r\in \r^{>0}$, $$\sup_{A\in M_{m,n}(\c)}\min\left(\|A\|\dotminus r, \inf_x \max(\|x\|\dotminus 1,(r\|x\|\dotminus \|Ax\|)\right)=0.$$
\item $\sup_{A,X,B} \left(\|f_{p,m,n,q}(A,X,B)\|_{p,q}\dotminus \left(\|A\|_{p,m}\cdot \|X\|_{m,n}\cdot \|B\|_{n,q}\right)\right)=0$.
\end{itemize}

The axioms for $T_{\operatorname{ops}}$ are the axioms for $T_{\operatorname{mns}}$ plus the following axioms:
\begin{itemize}
\item $\sup_{X,Y}|\|X\oplus Y\|_{m+p,n+q}-\max(\|X\|_{m,n},\|Y\|_{p,q})|=0$.
\end{itemize}

Then by Ruan's Theorem (see \cite[Theorem 13.4]{Paulsen}), we have:

\begin{thm}
Every (abstract) operator space $V$ is, in a natural way, a model $M(V)$ of $T_{\ops}$.  Conversely, every model of $T_{\opsys}$ is an operator space.  The map $V\mapsto M(V)$ is an equivalence of categories between the category of operator spaces with complete isometries and the category of models of $T_{\ops}$ with (model-theoretic) isomorphisms.
\end{thm}

\section{Definability of the operator norm}

In this appendix, we prove that the matrix norm on $M_n(A)$ is definable in the C$^*$ algebra $A$.  All of this material is due to Martino Lupini and we thank him for allowing us to include this material here.

Suppose that $A$ is a C$^*$ algebra.  We consider $A^n$ as a Hilbert $A$-module by giving it the usual $A$-module structure and
defining%
\begin{equation*}
\left\langle \vec x ,\vec y\right\rangle
=\sum x_{i}^{\ast }y_{i}
\end{equation*}%
In
particular the norm on $A^{n}$ is defined by%
\begin{equation*}
\left\Vert \vec x \right\Vert =\left\Vert \sum x_{i}^{\ast }x_{i}\right\Vert^{1/2} \text{.}
\end{equation*}

An element $\left( a_{ij}\right) $ of $M_{n}\left( A\right) $
defines an element $T_{\left( a_{ij}\right) }$ of $\B\left( A^{n}\right) $
by ordinary matrix multiplication.  The operator norm on $\B\left( A^{n}\right) $ defines a C*-norm on $M_n\left( A\right) $. This C*-norm has to agree with the usual (complete!)
C*-norm on $M_{n}\left( A\right) $.

Finally, set
$$X_{n,A}:=\{(\vec x,\vec y)\in A_1^{2n} \ : \ \max\left(\left\|\sum x_i^*x_i\right\|,\left\|\sum y_i^*y_i\right\|\right)\leq 1\}.$$
\begin{prop}
If $A$ is a C*-algebra and $\left( a_{ij}\right) \in M_{n}\left(
A\right) $ then%
\begin{equation*}
\left\Vert \left( a_{ij}\right) \right\Vert =\sup_{(\vec x,\vec y)\in X_{n,A}} \left\{ \left\Vert
\sum_{i\in n}\sum_{j\in n}x_{i}^{\ast }a_{ij}y_{j}\right\Vert\right\}.
\end{equation*}%
\end{prop}

\begin{proof}
First observe that if $x_1,\ldots,x_n\in A$ are such that
\begin{equation*}
\left\Vert \sum x_{i}^{\ast }x_{i}\right\Vert \leq 1,
\end{equation*}%
then for every $j$ we have%
\begin{equation*}
0\leq x_{j}^{\ast }x_{j}\leq \sum_{i\in n}x_{i}^{\ast }x_{i}\leq 1,
\end{equation*}%
whence $x_{j}$ belongs to the unit ball of $A.$
It now follows that
\begin{alignat}{2}
\left\Vert \left( a_{ij}\right) \right\Vert &=\left\Vert T_{\left( a_{ij}\right) }\right\Vert  \notag \\ \notag
&=\sup \left\{ \left\Vert T_{\left( a_{ij}\right) }\left( y_{k}\right)
\right\Vert \ : \  \left( y_{k}\right) \in A^{n}\text{, }\left\Vert
\left( y_{k}\right) \right\Vert \leq 1 \right\}  \\ \notag
&=\sup_{(\vec x,\vec y)\in X_{n,A}} \left\{ \left\Vert \left\langle \left( x_{h}\right) ,T_{\left(
a_{ij}\right) }\left( y_{k}\right) \right\rangle \right\Vert\right\}  \\ \notag
&=\sup_{(\vec x,\vec y)\in X_{n,A}} \left\{ \left\Vert
\sum_{i\in n}\sum_{j\in n}x_{i}^{\ast }a_{ij}y_{j}\right\Vert\right\}.
\end{alignat}%

In order to show that the operator norm on $M_n(A)$ is definable, we must show that $X_{n,A}$ is definable, whence quantifying over it preserves definability.  Towards this end, it suffices to show that the relation $\|\sum x_i^*x_i\|\leq 1$ is a stable relation.  Suppose that $x_1,\ldots,x_n\in A$ are such that $\left\Vert \sum x_{i}^{\ast }x_{i}\right\Vert \leq \left(
1+\varepsilon \right) ^{2}$.
Defining $y_{i}:=\frac{x_{i}}{1+\varepsilon }$, we see that $\|\sum_{i\in n}y_{i}^{\ast }y_{i}\|\leq 1$
and $\left\Vert y_{i}-x_{i}\right\Vert \leq \varepsilon$, as desired.
\end{proof}

\end{document}